%% file: DamVer_PaperI.tex
\documentclass[arxiv,reqno,bibliography=totoc,twoside,a4paper,12pt]{amsart}

\usepackage[utf8]{inputenc}

\usepackage{amsmath,bbm}
\usepackage{amsfonts}
\usepackage{amssymb}
\usepackage{amsthm}
\usepackage{float}
\usepackage{mathtools}
\usepackage{csquotes}
\usepackage{upgreek,esint,enumerate}
\usepackage[numbers]{natbib}
\usepackage{dsfont,ednotes}
\usepackage{enumitem}
\usepackage{graphicx}

\usepackage[colorlinks=true,linkcolor=blue,citecolor=blue, urlcolor=blue]{hyperref}
\usepackage{cleveref}

\usepackage[scaled]{helvet} 
\usepackage{courier} 
\usepackage[mathbf]{euler}
\usepackage[dvipsnames]{xcolor}
\usepackage{subcaption}

\usepackage[driver=pdftex,margin=3cm,heightrounded=true,centering]{geometry}

\usepackage{tikz-cd}
\usetikzlibrary{arrows,matrix,shapes,decorations.text, mindmap,shapes.misc,decorations.markings}

\usepackage{ifthen}
\usepackage{metalogo}
\usepackage{Styles/commands}
\usepackage{sepfootnotes}			
\newfootnotes{b}				
\input{footnotes}

\usepackage{standalone}				

\theoremstyle{plain}
\newtheorem{theo}{Theorem}[section]

\newtheorem{lem}[theo]{Lemma}
\newtheorem{cor}[theo]{Corollary}
\newtheorem{prop}[theo]{Proposition}
\newtheorem{defi}[theo]{Definition}
\newtheorem{rem}[theo]{Remark}

\usepackage[toc,page]{appendix}

\numberwithin{equation}{section}

\begin{document}

\title[Cauchy problem for the Dirac Operator on spacetimes]
{Cauchy problem for the Dirac Operator on spatially non-compact spacetimes}

\author{Orville Damaschke}
\address{Universit\"at Oldenburg,
26129 Oldenburg,
Germany}
\email{orville.damaschke@math.uni-oldenburg.de}

\date{\today}

\begin{abstract}
{Let $M$ be a globally hyperbolic manifold with complete spacelike Cauchy hypersurface $\Sigma$. We prove well-posedness of the Cauchy problem for the Dirac operator on globally hyperbolic manifolds with complete Cauchy hypersurfaces. This result is needed as preparation in showing a Fredholmness result in the manner, provided by Bär and Strohmaier, for certain non-compact Cauchy hypersurfaces in future work.}
\end{abstract}

\maketitle
\tableofcontents

\section{Introduction and statement of the main result}\label{chap:Intro}

This paper deals with the initial value problem of the Dirac equation on curved spacetime for the case of non-compact, but complete Cauchy hypersurfaces. The known results for the Cauchy problem with smooth inital data in \cite[Thm.4]{AndBaer} are extended to those with Sobolev regularity i.e. our contribution states that under (spatial) compactly supported initial data the resulting spinor fields are of finite energy, i.e. they are continuous in time and of some Sobolev regularity in space. This has been already announced but not concretely proven in \cite{AndBaer}. The well-posedness statement in this reference relies on the well-posedness of the Cauchy problem for the wave equation under similar regularity assumptions on the initial data and the inhomogeneity. We are going to present a direct proof. The results have been already non-peer-reviewed published in \cite{OD}. This existing version is only focused on the Cauchy problem, but we modify the content to the twisted Dirac operator and correct some issues. \\
\\
Our general setting is as follows: let $(M,\met)$ be a $(n+1)$-dimensional globally hyperbolic spin manifold and $E$ a Hermitian vector bundle over $M$; $n$ counts the number of spatial directions and is preferably chosen to be an odd integer. Any Cauchy hypersurface $\Sigma$ is a non-compact, but complete submanifold. $\SET{\Sigma_t}_{t\in \timef(M)}$ is a family of Cauchy hypersurfaces in $M$ which are level sets of the Cauchy temporal function $\timef$; $\timef(M)$ denotes the time domain of $M$. The analysis is focused on the (Atiyah-Singer) Dirac operator $\Dirac$ and its twisted version $\Dirac^E$. Any spin bundle $\spinb(M)$, associated to $TM$, is a bundle of modules over a Clifford bundle. As $M$ is even-dimensional, the spinor bundle decomposes into two subbundles $\spinb^{\pm}(M)$ such that $\Dirac$ separates into Dirac operators $D_{\pm}$, acting on $\spinb^{+}(M)$ respectively $\spinb^{-}(M)$. This also holds for the twisted spin bundle $\spinb_E(M):=\spinb(M)\otimes E$, which decomposes into the subbundles $\spinb^{\pm}_E(M):=\spinb^{\pm}(M)\otimes E$, and the twisted Dirac operator which decomposes into $D^E_{\pm}$. The spaces of finite $s$-energy spinors $FE^s_\scomp(M,\timef,D^E_{\pm})$ correspond to spatial compactly supported time-continuous sections of a bundle of $s$-Sobolev spaces on the time domain $\timef(M)$. They have compact support on each hypersurface and the images under $D^E_{\pm}$ are locally $L^2$ in time and Sobolev in space. The subspaces $FE^s_\scomp(M,\kernel{D^E_{\pm}})$ consist of kernel solutions in $FE^s_\scomp(M,\timef,D^E_{\pm})$. We refer to \clef{finensec} and \Cref{finensolkern} for the explicit definitions. Our first main result claims the well-posedness of the Cauchy problem by saying that $D^E_{\pm}$ and the restriction operator $\rest{t}$ for any $t \in \timef(M)$ are isomoporphisms between sections of finite energy spinors in $M$ and their initial values on $\Sigma$.
\begin{theo}\label{maintheo}
Let $(M,\met)$ be a $(n+1)$-dimensional globally hyperbolic spin manifold with complete, but possibly non-compact Cauchy hypersurface $\Sigma$, $E\rightarrow M$ a Hermitian vector bundle and $D_{\pm}^E$ the twisted Atiyah-Singer Dirac operators on spinors with positive respectively negative chirality; for all $s\in \R$ and any $t\in \timef(M)$ the mappings
\begin{eqnarray*}
\mathsf{res}_t \oplus D_{\pm}^E &:& FE^s_\scomp(M,\timef,D_{\pm}^E) \,\,\rightarrow\,\, H^s_\comp(\spinb^{\pm}_E(M)\vert_{\Sigma_t})\oplus L^2_{\loc,\scomp}(\timef(M),H^s_\loc(\spinb^{\mp}_E(M)_{\Sigma_\bullet}))\\
\text{and}\quad\quad\mathsf{res}_t  &:& FE^s_\scomp(M,\kernel{D_{\pm}^E}) \,\,\rightarrow\,\, H^s_\comp(\spinb^{\pm}_E(M)\vert_{\Sigma_t})
\end{eqnarray*}
are isomorphisms of topological vector spaces.
\end{theo} 
This is going to be proven as \Cref{inivpwell} and \Cref{homivpwell}. The first line can be rephrased in the manner of PDE analysis as follows: given an initial value $u_0 \in H^s_\comp(\spinb^{\pm}_E(M)\vert_{\Sigma_{t_0}})$ at any initial time $t_0 \in \timef(M)$ and inhomogeneity $f \in L^2_{\loc,\scomp}(\timef(M),H^s_\loc(\spinb^{\mp}_E(M)\vert_{\Sigma_\bullet})$, the Cauchy problems
\begin{equation*}
D^E_{ \pm} u = f \quad\text{with}\quad u\vert_{\Sigma_{t_0}}=\rest{t_0}(u)=u_0
\end{equation*}
each have a unique solution $u \in FE^s_\scomp(M,\timef,D^E_{\pm})$; the second line can be interpreted analogously with $u \in FE^s_\scomp(M,\kernel{D^E_{\pm}})$ for the homogeneous equation. This result generalises \cite[Thm.2.1]{BaerStroh} to non-compact, complete hypersurfaces as well as \cite[Thm.4]{AndBaer} to arbitrary Sobolev regularity. The operational formulation of this well-posedness result suggests the existence of a wave evolution operator $Q(t_2,t_1)$, transporting solutions along one Cauchy hypersurface at time $t_1$ to another solution at time $t_2$. We show that its nature as Fourier integral operator of order zero with canonical relation, related to the lightlike cogeodesic flow between the two Cauchy hypersurfaces, carries over to the non-compact setting. \\ 
\\
The paper is organised as follows: \Cref{chap:Back} is meant for fixing important notations and to recapitulate the geometric setup of globally hyperbolic manifolds, function spaces on manifolds in general with a closer look on Sobolev spaces, and the spaces of relevance for the proof of \Cref{maintheo}. \Cref{chap:Dirac} deals with Spin-Dirac operators and their representation along Cauchy hypersurfaces. \Cref{chap:Cauchy} is focused on the proof of \Cref{maintheo}. An energy estimate and its consequences are derived in the first subsection as important intermediate step. The well-posedness of the homogeneous Dirac equation implies a wave evolution operator mapping compactly supported Sobolev sections over one hypersurface to compactly supported Sobolev-sections over another hypersurface. This operator is again well defined as Fourier integral operator between non-compact hypersurfaces. This is investigated in \Cref{chap:Feynman}. The appendix of this paper contains a closer look on this operator class and their application in hyperbolic initial value problems.\\
\\
We will consider these two results in a series of coming papers and specify them to the setting that the Cauchy hypersurfaces are Galois coverings. We are able to show with similar arguments as in \cite{BaerStroh} that the wave evolution operators and some of its spectral decomposits are $L^2$-Fredholm in the von Neumann algebra setting of bounded operators, commuting with the left action representation of the group of deck transformations. This leads to the wanted $L^2$-Fredholmness of the Lorentzian Dirac operators. This second paper will contain all necessary background of Galois coverings, von Neumann algebras, and some useful preparatory statements beforehand. It corrects and modifies the content of the second part in \cite{OD}. A concrete index formula in this setting will be derived in a third paper. All of these three papers are mainly taken from the authors thesis \cite{ODT}. This upcoming work as well several other so far known contributions to index theory on globally hyperbolic manifolds (e.g. generalised boundary conditions in \cite{BaerHan}, non-compact Cauchy hypersurfaces for Dirac operators of strongly Callias type in \cite{Braverman2020}, non-compact time domain in \cite{shenwrochna}, a local index theorem and Fredholmness for non-self-adjoint Dirac operators in \cite{BaerStroh2}) are based on the pioneering work of Bär and Strohmaier in \cite{BaerStroh} and the starting point of all these modifications is the well-posedness of the Cauchy problem for the non-elliptic Dirac operator on globally hyperbolic spacetimes.\\ 
\\
\textit{Acknowledgements.} The author would like to thank Boris Vertman and Daniel Grieser for their supervising and topical support in the process of working out these and the following results. Moreover, I would like to thank Alexander Strohmaier and Christian Bär who encouraged the author to extend their work, contributing to the increasing research field of index theory on Lorentzian spaces and its applications. 

\section{Notations, geometric setup \& known facts}\label{chap:Back}

We recapitulate notations and some useful facts of sections of vector bundles over a manifold and operators, acting between those sections. After introducing some basic background of globally hyperbolic manifolds, we focus on several function spaces especially constructed for those spacetimes. In order to investigate Sobolev regularity on spacelike non-compact Cauchy hypersurfaces, a repetition of some basics about Sobolev spaces on non-compact Riemannian manifolds are prepared.

\subsection{Function spaces, operators, Sobolev spaces for manifolds}\label{chap:Back-sec:Funcspace}

Let $E\,\rightarrow \, M$ be any $\C-$(anti-)linear smooth vector bundle over a $n$-dimensional manifold $M$. We moreover assume $M$ to be a pseudo-Riemannian manifold with metric $\met$, which is a totally symmetric smooth section of $(T^{*}M)^{\otimes 2}$, and equip the tangent bundle $TM$ and the cotangent bundle $T^{*}M$ with the Levi-Civita connection. The vector bundle is assumed to come with a related Koszul connection $\Nabla{E}{}$. We denote the space of smooth sections of $E$ with $C^\infty(M,E)$; we write $C^\infty(E)$ if the manifold is clear from the context or the notation of the vector bundle. If the support is contained in a closed subset $A \subset M$, we write $C^\infty_A(M,E)$. The union of $C^\infty_K(M,E)$ over all compact subsets $K$ of $M$ defines $C^\infty_{\comp}(M,E)$ as \textit{space of compactly supported sections} of $E$ with continuous inclusion mapping $C^\infty_K \hookrightarrow C^\infty_{\comp}$. Any linear map between this space and other locally convex topological vector spaces is continuous if it is continuous on the restriction to $C^\infty_K(M,E)$. We assume that $M$ is orientable such that a volume form $\dvol{}$ on $M$ exists. If the vector bundle $E$ comes with a bundle metric $\idscal{1}{E}{\cdot}{\cdot}\,$, which is a pointwise sesquilinear form\bnote{f3} $\idscal{1}{E_p}{\cdot}{\cdot}:E_p\times E_p\rightarrow \C$, another sesquilinear form
\begin{equation}\label{sesquismoothcomp}
C^\infty_\comp(M,E)\times C^\infty_\comp(M,E) \ni(u,v)\,\mapsto\,\idscal{1}{C^\infty_\comp(M,E)}{u}{v}:=\int_M \idscal{1}{E_p}{u}{v} \,\dvol{}(p) 
\end{equation}
can be introduced. If the vector spaces $E_p$ for each $p\in M$ are equipped with an inner product $\dscal{1}{E}{\cdot}{\cdot}(p):E\times E\rightarrow \C$, \clef{sesquismoothcomp} becomes positive definite if we replace $\idscal{1}{E_p}{\cdot}{\cdot}$ with $\dscal{1}{E_p}{\cdot}{\cdot}$. We also write $\idscal{1}{E}{\cdot}{\cdot}(p)$ for $\idscal{1}{E_p}{\cdot}{\cdot}$ and similarly $\dscal{1}{E}{\cdot}{\cdot}(p)$ for $\dscal{1}{E_p}{\cdot}{\cdot}$. If the vector bundle does not come with a further structure, we can consider the anti-dual vector bundle $\overline{E}^{*}\,\rightarrow\,M$ such that a pointwise dual pairing $\dpair{1}{E}{\cdot}{\cdot}(p)=\dpair{1}{E_p}{\cdot}{\cdot}:\overline{E}^\ast_p \times E_p \rightarrow \C $ can be introduced. Integrating over the manifold yields a distributional pairing for smooth and compactly supported sections of $E$:
\begin{equation}\label{dpairreg}
(\psi,\phi)\,\mapsto\,\dpair{1}{C_\comp^\infty(M,E)}{\psi}{\phi}:= \int_M \dpair{1}{E}{\psi}{\phi}(p)\,\dvol{}(p)
\end{equation}
for $\psi \in C_\comp^\infty(M,\overline{E}^\ast)$ and $\phi \in C^\infty_\comp(M,E)$. If the bundle metric is Riemannian or Hermitian, \clef{sesquismoothcomp} induces an $L^2$-inner product
\begin{equation}\label{L2prodgeneralfunc}
\dscal{1}{L^2(M,E)}{u}{v}:=\int_M \dscal{1}{E}{u}{v}(p) \,\dvol{}(p) 
\end{equation}
for $u,v$ as in \clef{sesquismoothcomp} and $\dscal{1}{E}{\cdot}{\cdot}(p)$ positive definite. This inner product induces a $L^2$-norm $\norm{u}{L^2(M,E)}^2:=\dscal{1}{L^2(M,E)}{u}{u}$. The completion of $C^\infty_\comp(M,E)$ with respect to this norm defines the Hilbert space $L^2(M,E)$ of \textit{square-integrable sections} of $E$. We denote spaces of \textit{compactly supported} or rather \textit{local} $L^2$-\textit{sections} with $L^2_\comp(M,E)$ respectively $L^2_\loc(M,E)$.\\
\\
As long as the intersection of supports of the two sections $\phi$ and $\psi$ is compact, the dual pairing \clef{dpairreg} still makes sense. A \textit{distributional section} of $E$ is a $\C$-linear functional on test sections $\phi \in C^\infty_\comp(M,\overline{E}^{*})$ with respect to a dual or rather distributional pairing $\dpair{1}{C^\infty_\comp(M,\overline{E}^{*})}{u}{\phi}$. Hence the space of distributional sections is the dual space of smooth sections with compact support and we designate the space with $C^{-\infty}(M,E)$. We define the space $C^{-\infty}_\comp(M,E)$ of compactly supported distributions in the same way as the dual space $(C^\infty(M,\overline{E}^\ast))^\ast$.\\ 
\\
Let $F\rightarrow N$ be another $\C$-(anti-)linear vector bundle over a manifold $N$ of possibly different dimension; the ranks of the vector bundles are $m_E$ for $E$ and $m_F$ for $F$. 
Given a linear operator $P: C^\infty_\comp(M,E)$ to $C^{-\infty}(N,F)$. We denote with $P^\dagger$ the \textit{formal dual operator} as map from  $C^\infty_\comp(N,\overline{F}^\ast)$ to $C^{-\infty}(M,\overline{E}^\ast)$. By this, $P$ extends to a linear operator $C^{-\infty}(M,E)$ to $C^{-\infty}(N,F)$ by applying $P^\dagger$ on the test section under the dual pairing if $P^\dagger$ maps from  $C^\infty_\comp(N,\overline{F}^\ast)$ to  $C^\infty_\comp(M,\overline{E}^\ast)$. If the dual operator maps from $C^\infty_\comp(M,E)$ to $C^\infty_\comp(N,F)$ and is taken with respect to the inner product $\dscal{1}{L^2(M,E)}{\cdot}{\cdot}$, we designate it with $P^\ast$ and call it \textit{formal adjoint operator}.\\
\\
If the operator $P$ is continuous, the Schwartz Kernel Theorem induces a bijective correspondence between $P$ and a \textit{Schwartz kernel} $K$ which is a distribution on the Cartesian product $N\times M$, such that the action of $P$ on a function $\phi$ can be represented as
$$ (P\phi)(x)=\int_M K(x,y)u(y) \dvol{}(y)$$
in the distributional sense. The operator can be characterised by the properties of its kernel and vice versa. The dual operator has a Schwartz kernel which is an element in $C^{-\infty}(M\times N,\Hom(\overline{F}^\ast,\overline{E}^\ast))$. An operator $P$ with Schwartz kernel $K$ is said to be \textit{properly supported} if both projections $\pi_N:N\times M\,\rightarrow\, N$ and $\pi_M:N\times M \rightarrow M$ are proper maps on $\supp{K}\subset N\times M$. A more handy characterisation is the following: the operator $P$ is properly supported if and only if
\begin{equation}\label{propsupp}
\begin{split}
(1)&\quad\forall\,K_M\Subset M\,\exists\, K'_N \Subset N:\,\supp{u}\subset K_M \,\Rightarrow \, \supp{Pu}\subset K'_N\quad, \\
(2)&\quad\forall\,K_N\Subset N\,\exists\, K'_M \Subset M\,:\,\supp{v}\subset K_N \,\Rightarrow \, \supp{P^\dagger u}\subset K'_M \quad.
\end{split}
\end{equation}
In a nutshell, an operator is properly supported if and only if it maps compactly supported sections to compactly supported sections. The composition $Q\circ P$ of two operators then becomes well-defined if at least $P$ is properly supported. Moreover, the composition of properly supported operators is again properly supported.\\
\\
We assume that the reader is familiar with the background of differential and pseudo-differential operators in the manifold setting and the notion of the principal symbol. The standard, non-geometric introduction of these operators is based on local coordinates on a neighbourhood of $M$ and two trivialisations of the vector bundles $E,F$ over one and the same manifold $M$. We call $P$ a \textit{differential operator of order k} if for all such open neighbourhoods and trivialisations the operator $P$ can be related to a $(m_F\times m_E)$-matrix of linear differential operators of order $k$. We denote the set of those linear partial-differential operators of order $k$ with $\Diff{k}{}(M,\Hom(E,F))$ and write $\Diff{k}{}(E,F)$ if the manifold is clear from the context; in addition, we set $\Diff{k}{}(E)$ for $\Diff{k}{}(M,\End(E))$. Since differential operators are local, they are automatically properly supported. $P$ is called \textit{pseudo-differential operator of order $r\in \R$} if for every open neighbourhood and any trivialisation each matrix entry is given as a pseudo-differential operator of order $r$. We designate the set of pseudo-differential operators of order $r$ with $\ydo{r}{}(M,\Hom(E,F))$ and also use the abbreviations $\ydo{r}{}(E,F)$ respectively $\ydo{r}{}(E)$ as like for differential operators. We denote with $\ydo{r}{\mathsf{prop}}(M,\Hom(E,F))$ the subset of properly supported pseudo-differential operators. Fourier integral operators as another important class are presented in \Cref{chap:Fourier} with some details.\\
\\
Let $E \rightarrow M$ be a Hermitian vector bundle over a Riemannian manifold $M$ with connection $\Nabla{E}{}$. Any real power $s\in\R$ of the Laplace-type operator $\Lambda^2:=(\Nabla{E}{})^\ast\Nabla{E}{} + \Iop{E}$ is a properly supported and elliptic pseudo-differential operator with strictly positive principal symbol. If the manifold $M$ is in addition compact or complete, it becomes essentially self-adjoint on $L^2(M,E)$ with positive spectrum (see i.e. \cite{yoshida}). We write $\Lambda^s$ for $(\Lambda^2)^{s/2}$. This construction becomes important in introducing Sobolev spaces for any real Sobolev degree $s\in \R$. 
We start with Sobolev spaces on compact manifolds $M$ without boundary as they provide a bulding block for the non-compact case. The $s-$\textit{Sobolev norm} of $u\in C^\infty(M,E)$ is defined as
\begin{equation}\label{sobnorm}
\Vert u \Vert_{H^s(M,E)}:=\Vert \Lambda^s u \Vert_{L^2(M,E)}
\end{equation}
and the norm completion of $C^\infty(M,E)$ with respect to \clef{sobnorm} defines the \textit{Sobolev space of order} $s$ which we designate with $H^s(M,E)$. The definition depends neither on the choice of the metric nor of the connection such that all Sobolev norms and spaces with different metrics and connections are equivalent. \\
\\
For non-compact manifolds, we follow the more practical introduction, presented in \cite[Sec.1.6]{BaerWafo}. There, Sobolev sections of $E$ with fixed compact support in $K$ in a non-compact Riemannian manifold $M$ are reinterpreted as Sobolev sections on another now closed Riemannian manifold $\widetilde{M}$ which is constructed as closed double of a compact superset of $K$ with smooth and totally geodesic boundary. The vector bundle $E\vert_K$, its connection and bundle metric as well as the Riemannian metric then extend smoothly to $\widetilde{M}$ with respect to the extended vector bundles $\widetilde{E}$, satisfying $\widetilde{E}\vert_K=E\vert_K$, respectively $T\widetilde{M}$. Any $u\in C^\infty_K(M,E)$ then can be viewed as a section $u\in C^\infty(\widetilde{M},\widetilde{E})$ such that the space of Sobolev sections of order $s\in \R$ with fixed compact support in $K$ can be defined via
\begin{equation}\label{Hcompfix2}
H^s_K(M,E):= \overline{C^\infty_K(M,E)}^{\Vert\cdot\Vert_{H^s(\widetilde{M},\widetilde{E})}} 
\end{equation}
where $H^s(\widetilde{M},\widetilde{E})$ is the Sobolev space of sections over the compact double, defined as completion of $C^\infty(\widetilde{M},\widetilde{E})$ with respect to \clef{sobnorm}. The advantage of this definition is that Sobolev sections with compact support can be treated in the same way as Sobolev sections on a closed manifold with all beneficial properties. The space $H^m_{\comp}(M,E)$ of compactly supported Sobolev sections then follow by taking the union over all compact subsets. Local Sobolev sections are regarded as those distributional sections $u$ such that $\phi u \in H^s_{\supp{\phi}}$ for all smooth and compactly supported functions $\phi$:
\begin{equation}\label{Hloc}
H^m_{\loc}(M,E):=\SET{u \in C^{-\infty}(M,E)\,\vert\,\phi u \in H^m_{\comp}(M,E) \quad \forall\,\phi \in C^\infty_\comp(M)}\quad.
\end{equation}
If $K\subset K'$, one has the inclusion $C^\infty_K(M,E)\subset C^\infty_{K'}(M,E)$, inducing $H^s_K(M,E)\subset H^s_{K'}(M,E)$ as continuous linear inclusion. This implies the continuous inclusion $H^s_K(M,E)\hookrightarrow H^s_\comp(M,E)$. All introduced Sobolev spaces are provided with important properties such as regularity of Sobolev functions after restriction to submanifolds, embeddings or localisation which can be found e.g. in \cite[pp.57-65]{shubin11}.

\subsection{Globally hyperbolic manifolds}\label{chap:Back-sec:Globhyp}

We now focus our attention onto a $(n+1)$-dimensional time-oriented Lorentzian manifold $M$ with Lorentzian metric $\met$; the signature chosen throughout this and the following papers is $(-,+,+,\dots,+)$. The letter $n \in \N$ then counts the number of spatial dimensions. Further details of Lorentzian geometry in general can be found in several references, see e.g. \cite{ONeill}.\\
\\
A subset $\Sigma$ of $M$ is called \textit{Cauchy hypersurface} if $\Sigma$ is a smooth, embedded hypersurface, and every inextendable timelike curve in $M$ meets $\Sigma$ exactly once. If $M$ admits several Cauchy hypersurfaces, then all of them are homeomorphic to each other. A time-oriented Lorentzian manifold is called \textit{globally hyperbolic} if and only if it contains a Cauchy hypersurface; see \cite[Thm.11]{Geroch70}. Geroch's topological splitting theorem says that any globally hyperbolic manifold $M$ is homeomorphic to $\R\times \Sigma$. The following result shows that the manifold is even isometrically related to this product manifold:
\begin{theo}[\textit{Geroch's splitting theorem}, Theorem 1.1 in \cite{BerSan2005} \& \cite{Geroch70}]\label{theo22-1}
Suppose $(M,\met)$ is a globally hyperbolic manifold and $\Sigma$ a spatial Cauchy hypersurface; the following holds:
$(M,\met)$ is isometric to the product manifold $\R\times\Sigma$ with Lorentzian metric 
\begin{equation*}
\met = - N^2 \differ \timef^{\otimes 2} + \met_\timef \quad  
\end{equation*} 
where $\timef$ is a surjective smooth function on $M$, $N\in C^\infty(M,\Rpos)$ and $\met_\timef$ is a smooth one-parameter family of smooth Riemannian metrics on $\Sigma$, satisfying
\begin{itemize}
\item[(1)] $\mathrm{grad}(\timef)$ is a past-directed timelike gradient on $M$,
\item[(2)] each hypersurface $\Sigma_t$ for $t\in \timef(M)\subset \R$ is a spacelike Cauchy hypersurface with Riemannian metric $\met_\timef$, where $\Sigma_{t_0}:= \Sigma$, and
\item[(3)] $\Span{}{\mathrm{grad}\vert_p(\timef)}$ is orthogonal to $T_p \Sigma_t$ with respect to $\met_\timef\vert_p$ at each $p \in \R\times\Sigma$.
\end{itemize}
\end{theo}
$\timef$ is referred to as \textit{Cauchy temporal function}. The time domain for a certain globally hyperbolic manifold is denoted with $\timef(M)$. The existence of such a function ensures that each level set $\Sigma_t$ can be interpreted as a slice $\SET{t}\times\Sigma$. \Cref{theo22-1} furthermore suggests that along $\mathrm{grad}(\timef)$ the spacetime is foliated by these level sets wherefore the function $N$ is called \textit{lapse function (of the foliation)}. If $\timef(M)$ does not contain any critical points of $\timef$, then each level set is regular and the regular level set theorem ensures that $\Sigma_t$ is a closed embedded submanifold of codimension one. Each embedding $\inclus_t:\Sigma_t\,\hookrightarrow\,M$ becomes a proper map. In the following we will use $\differ t$ and $\partial_t$ instead of $\differ \timef$ and $\mathrm{grad}(\timef)$ to stress the time differentials/derivatives as coordinate (co-)vector with respect to a hypersurface $\Sigma_t$ for $t\in \timef(M)$. Hence and henceworth, we will rewrite the metric of a globally hyperbolic manifold as
\begin{equation}\label{orthmet}
\met = - N^2 \differ t^{\otimes 2} + \met_t \quad . 
\end{equation} 
We also need the notion of causal sets: for any $p \in M$ define
\begin{equation*}
\begin{split}
\Jlight{+}(p)&:=\SET{q \in M \,\vert\, \exists\,\ \text{causal future-directed curve}\,\, \upgamma : p \,\rightsquigarrow\, q} \quad \text{and} \\
\Jlight{-}(p)&:=\SET{q \in M \,\vert\, \exists\,\ \text{causal past-directed curve}\,\, \upgamma : p \,\rightsquigarrow\, q} \quad .
\end{split}
\end{equation*}
For any subset $A \subset M$ put $\Jlight{\pm}(A):=\bigcup_{p \in A}\Jlight{\pm}(p)$ as \textit{future} respectively \textit{past light cone} of $A$. The \textit{causal domain}, \textit{domain of influence} or \textit{light cone} of $A$ is the union $\Jlight{}(A):=\Jlight{+}(A)\cup\Jlight{-}(A)$.\\ 
\begin{figure}[H]
\centering
\begin{subfigure}[b]{0.48\textwidth}
\centering
\includegraphics[width=0.65\textwidth]{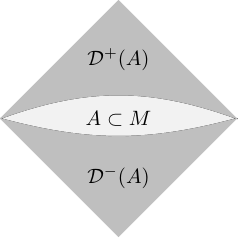}
\caption{}
\end{subfigure}
\,
\begin{subfigure}[b]{0.48\textwidth}
\centering
\includegraphics[width=\textwidth]{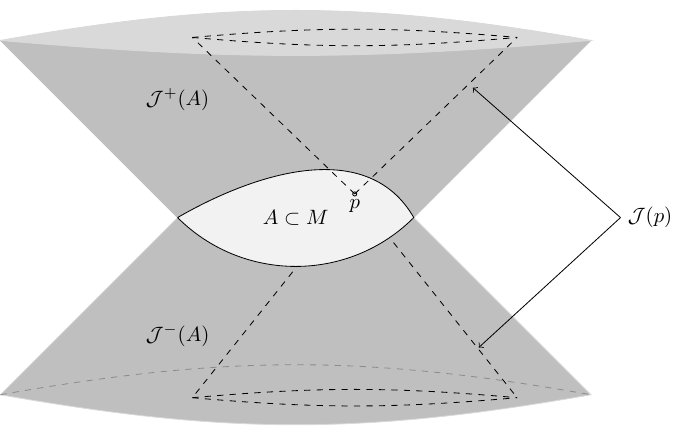}
\caption{}
\end{subfigure}
\caption{Domain of dependence of an achronal subset $A\subset M$ in (A) and domain of influence for any subset $A\subset M$ in (B).}
\end{figure}
The light cone concept allows to rephrase global hyperbolicity as follows: a time-oriented Lorentzian manifold $(M,\met)$ is globally hyperbolic if an only if it is causal, strongly causal, and $\Jlight{+}(p)\cap\Jlight{-}(q)$ is compact for all $p,q \in M$. In comparison, the concept of the \textit{domain of dependence}, \textit{causal diamond} or \textit{Cauchy developement} is defined as $\domdep{}(A)=\domdep{+}(A)\cup \domdep{-}(A)$ for $A \subset M$ such that $A$ is achronal, where
\begin{equation*}
\begin{split}
\domdep{+}(A)&:=\SET{p \in M\,\vert\,\text{every past inextendible causal curve through}\,\,p\,\,\text{meets}\,\,A}\,\text{and} \\
\domdep{-}(A)&:=\SET{p \in M\,\vert\,\text{every future inextendible causal curve through}\,\,p\,\,\text{meets}\,\,A}
\end{split}
\end{equation*}
are the \textit{future} are respectively \textit{past domain of dependence} of a subset $A$. It satisfies $A \subset \domdep{\pm}(A)\subset\Jlight{\pm}(A)$ and a globally hyperbolic manifold can be depicted as causal diamond of its Cauchy hypersurface $\Sigma$: $M=\domdep{}(\Sigma)$.\\ 
\\
A subset $A \subset M$ is called \textit{spatially/spacelike compact} if $A$ is a closed subset and there exists a compact subset $K \subset M$ such that $A \subset \Jlight{}(K)$. The intersection of a spatially compact subset with any Cauchy hypersurface is compact. One calls in contrast to this definition the whole manifold $M$ spatially compact if every Cauchy hypersurface of $M$ is compact. A closed subset $A \subset M$ is \textit{future/past compact} if $A\cap \Jlight{\pm}(K)$ is compact for every compact $K \subset M$; it is called \textit{temporal/timelike compact} if $A$ is both future and past compact. In contrast, we call the whole manifold $M$ temporal compact if $\timef(M)$ is a closed interval. This is equivalent by saying that there exist $t_1,t_2 \in \R$ such that $\timef(M)=[t_1,t_2]$. $M$ is then viewed as the causal diamond $\Jlight{+}(\Sigma_1)\cap\Jlight{-}(\Sigma_2)$ for $\Sigma_{1}=\Sigma_{t_1}$ and $\Sigma_{2}=\Sigma_{t_2}$. We will also use this terminology to express that we restrict the possibly non-compact time domain of $M$ to any compact time interval $[t_1,t_2]$. Thus, any restriction $M\vert_{[t_1,t_2]}$ becomes temporal compact in the original sense.\\
\begin{figure}[H]
\includegraphics[width=0.65\textwidth]{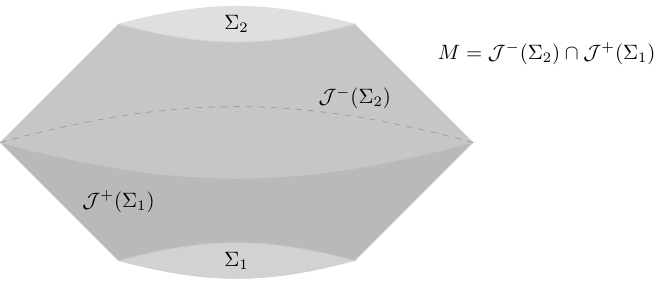}
\caption{A temporal compact globally hyperbolic manifold $M$.}
\end{figure}
We slightly generalise the situation and consider $\Sigma$ as Riemannian hypersurface in an ambient manifold $M$ which is isometric to $\R\times \Sigma$ (\textit{topological cylinder manifold}). We consider the metric 
\begin{equation}\label{onemetric}
\met^{[\epsilon]}:=\epsilon N^2\differ t^{\otimes 2} + \met_t
\end{equation}
for $\epsilon\in \SET{\pm 1}$. $N$ is again the lapse function and $\SET{\met_t}_{t\in \R}$ with $\met_t:=\met_{\Sigma_t}$ is a smooth one-parameter family of Riemannian metrics on $\Sigma$. For $\epsilon=-1$, we are in the Lorentzian case with $\met^{[-1]}=\met$ such that $M$ represents a globally hyperbolic manifold. For $\epsilon=+1$ we are in the Riemannian case which we will denote by $\check{\met}:=\met^{[+1]}$ for later use and is referred to as the \textit{flipped metric} of \clef{orthmet}. Let $\SET{e_i(t)}_{i=1}^n$ be a local tangent frame on the slice $\Sigma_t$ which is Riemann-orthonormal with respect to $\met_t$. We can lift this basis to a local frame $\SET{e_0(t)}\cup\SET{e_i(t)}_{i=1}^n$ in $M$ which is orthonormal with respect to $\met^{[\epsilon]}$ for each $t\in \R$. This implies that $e_0(t)$ has to be perpendicular to each slice. We construct $e_0$ to be parallel to $(-\partial_t)$ and we conclude from the orthonormality assumption that $e_0(t)=-\frac{1}{N}\partial_t$. Because of its often appearance this vector gets from now on the designation $\mathsf{v}$ and it comes with several important properties.
\begin{lem}[cf. Proposition 4.1 (14) in \cite{BaerGauMor}]\label{lem2-1}
For $p \in \Sigma_t$ let $X,Y \in T_p\Sigma_t$; the following expressions for $\mathsf{v}$ are fulfilled near the point $p$:
\begin{itemize}
\item[(1)] $\mathsf{v}$ is autoparallel with respect to $\nabla$, i.e. $\nabla_{\mathsf{v}}\mathsf{v} = 0$, and
\item[(2)] $\met^{[\epsilon]}(\wein(X),Y)=\frac{\epsilon}{2 N}\partial_t \met_t(X,Y)$ with $\wein$ as Weingarten maps.
\end{itemize} 
\end{lem}
\begin{proof}
The proof works as for the general considerations in the reference, where the used \textit{Weingarten map} has to be specified to
\begin{equation*}\label{weingarteneqII}
\wein(X)=-\epsilon\nabla_{X}\mathsf{v} \quad .\qedhere
\end{equation*}

\end{proof}
The Gauss formula of the tangent bundle covariant derivative along $\Sigma_t$ can be expressed with the Weingarten map \clef{weingarteneqII} and we can calculate the Christoffel symbols of the connection in terms of the data along the hypersurface: let latin indices indicate tangential or spacelike coordinates and "$0$" the normal or timelike direction; the symbols are given as follows:
\begin{equation}\label{chrissymb}
{\Gamma}^M_{jk,l}=\Gamma^\Sigma_{jk,l}\quad,\,{\Gamma}^M_{jk,0}=\met^{[\epsilon]}(\wein(e_j),e_k)=-{\Gamma}^M_{j0,k}\quad\text{and}\quad\Gamma^M_{j0,0}=\Gamma^M_{00,0}=\Gamma^M_{00,l}=0\quad.
\end{equation}
We denote with $\dvol{\Sigma_t}$ the induced volume form on each slice $\Sigma_t$ which in local coordinates $\SET{x^i}_{i=1}^n$ for $\Sigma$ is given by 
\begin{equation}\label{int0}
\dvol{\Sigma_t}=\sqrt{\det\left(\met_t\right)} \bigwedge_{i=1}^n \differ x^{i} \quad.
\end{equation}
Expressions of the form
\begin{equation}\label{int1}
I_{f}(t)=\int_{\Sigma_t} f_t \dvol{\Sigma_t} \quad
\end{equation}
for any integrable and $t$-differentiable function $f_t$ on $\Sigma_t$ can be differentiated by doing variation with compact support:
\begin{equation}\label{variationtIntegral}
\left. \frac{\differ}{\differ t} I_f(t)\right\vert_{\tau} = -\int_{\Sigma_\tau} \phi(p) \biggl(\epsilon n H_\tau(p)+\left(\mathsf{v} f_{t}\right)\vert_{\tau,p}\biggr)  \dvol{\Sigma_\tau}(p) \quad 
\end{equation}
for $\phi \in C^\infty_c(\Sigma_\tau)$ and $H_{\tau}$ the mean curvature of $\Sigma_\tau$. This technicality can be proven with the same techniques used for the (Lorentzian) first variation of area formula.

\subsection{Function spaces on globally hyperbolic manifolds}\label{chap:Back-sec:Funcglob}

We recall some special function spaces on a globally hyperbolic manifold $M$ with complete, boundaryless Cauchy hypersurface $\Sigma$ and Cauchy temporal function $\timef:M\rightarrow \R$. Further details and more examples of function spaces on globally hyperbolic manifolds are presented in \cite[Sec.1.7]{BaerWafo} and \cite[Chap.2]{Baergreen}.
\newpage
\noindent Given $s\in \R$ and the family $\SET{H^s_\loc(E\vert_{\Sigma_t})}_{t \in \timef(M)}$ as Fr\'{e}chet bundle over $\timef(M)\subset \R$. The slices differ from each other only in the metric $\met_t$, but not topologically. Since different metrics lead to equivalent Sobolev norms, each Sobolev space $H^s_K(E\vert_{\Sigma_t})$ for $K\Subset \Sigma$ and consequently $H^s_\comp(E\vert_{\Sigma_t})$ and $H^s_\loc(E\vert_{\Sigma_t})$ are equivalent. We keep the extra $t$ to mark the different metrics and furthermore keep notational compatability with the referred literature. $\SET{H^s_\loc(E\vert_{\Sigma_t})}_{t \in \timef(M)}$ can be globally trivialised as follows: for each $t \in \timef(M)$ a section of this bundle becomes a section in $H^s_\loc(E\vert_{\Sigma_t})$. Sections of this bundle can be moved to $H^s_\loc(E\vert_{\Sigma_{\tau}})$ for a fixed $\tau \in \timef(M)$ by parallel transport along the integral curves of the vector field $\mathrm{grad}(\timef)$. The support properties and the Sobolev regularity are preserved by this transport such that this bundle of Fr\'{e}chet spaces becomes diffeomorphic to $\timef(M)\times H^s_\loc(E\vert_{\Sigma_{\tau}})$; $l$-times continuously differentiable sections in the time parameter ($l \in \N_0$) are denoted  by $C^l(\timef(M),H^s_\loc(E\vert_{\Sigma_{\bullet}}))$. The elements can be considered as regular distributions on $M$ and are locally integrable for $s\geq 0$ since $H^s_\loc \subset L^2_\loc$ by the continuous embedding of Sobolev spaces. Moreover, $C^l(\timef(M),H^s_\loc(E\vert_{\Sigma_{\bullet}}))$ is embedded into $C^{-\infty}(M,E)$. One defines for any compact subinterval $I \subset \timef(M)$ and any spatially compact $K \subset M$ the space
\begin{equation}\label{scompsuppglobdiff}
C^l_K(\timef(M),H^s_\loc(E\vert_{\Sigma_{\bullet}})):=\Bigl\{ u \in C^l(\timef(M),H^s_\loc(E\vert_{\Sigma_{\bullet}}))\,\Big\vert\,\supp{u}\subset K\Bigr\}
\end{equation}
with the seminorm
\begin{equation}\label{snclk}
\Vert u \Vert_{I,K,l,s}:= \max_{k \in [0,l]\cap \N_0} \max_{t \in I} \Big\Vert (\nabla^{E}_t)^{l} u\Big\Vert_{H^s_\loc(E\vert_{\Sigma_t})} \quad.
\end{equation}
Varying over all compact subsets $I \subset \timef(M)$ shows for fixed $l$, $K$, and $s$ that \clef{scompsuppglobdiff} is a Fr\'{e}chet space. Taking the union over all spatially compact subset defines sections of this bundle which have support in any spatially compact subset of $M$:
\begin{equation}\label{clsc}
C^l_{\scomp}(\timef(M),H^s_\loc(E\vert_{\Sigma_{\bullet}})):=\bigcup_{\substack{K \subset M \\ K\,\,\text{spatially}\\ \text{compact}}}C^l_K(\timef(M),H^s_\loc(E\vert_{\Sigma_{\bullet}})) \quad.
\end{equation}
The inclusion $C^l_K(\timef(M),H^s_\loc(E\vert_{\Sigma_{\bullet}}))\hookrightarrow C^l_\scomp(\timef(M),H^s_\loc(E\vert_{\Sigma_{\bullet}}))$ is continuous and any linear map from $C^l_\scomp(\timef(M),H^s_\loc(E\vert_{\Sigma_{\bullet}}))$ to any locally convex topological vector space is continuous if and only if the restriction to $C^l_K(\timef(M),H^s_\loc(E\vert_{\Sigma_{\bullet}}))$ is continuous for any spatially compact subset $K$. The case $l=0$ will be of special interest.
\begin{defi}
Let $E \rightarrow M$ be a vector bundle over a globally hyperbolic manifold $M$ with temporal function $\timef$ and foliating family of spatial Cauchy hypersurfaces $\SET{\Sigma_t}_{t \in \timef(M)}$; for any $s \in \R$ we set
\begin{equation}\label{finensec}
FE^s_{\scomp}(M,\timef,E):=C^0_{\scomp}(\timef(M),H^s_\loc(E\vert_{\Sigma_{\bullet}}))
\end{equation}
to be the \textit{space of finite} $s$-\textit{energy sections}.
\end{defi}
The continuity with respect to the time coordinate can be weakened to local square-integrability.
\begin{defi}\label{L2locscK}
Let $E \rightarrow M$ be a vector bundle over a globally hyperbolic manifold $M$ with temporal function $\timef$ and foliating family of spatial Cauchy hypersurfaces $\SET{\Sigma_t}_{t \in \timef(M)}$ and $K \subset M$ spatially compact; the space $L^2_{\loc,K}(\timef(M),H^s_\loc(E\vert_{\Sigma_{\bullet}}))$ consists of those sections $u$ such that
\begin{itemize}
\item[(a)] $\supp{u} \subset K\cap\Sigma_t$ for almost all $t\in \timef(M)$ ;
\item[(b)] $t\,\mapsto\,\dpair{1}{C^\infty_\comp((E\vert_{\Sigma_t})^{*})}{u}{\phi\vert_{\Sigma_t}}$ is measurable for any $\phi\in C^\infty_{\comp}(M,E)$ and
\item[(c)] $t\,\mapsto\,\Vert u \Vert_{H^s_\loc(E\vert_{\Sigma_t})}$ is in $L^2_\loc(\timef(M))$
\end{itemize}
for any $s \in \R$.
\end{defi}
One can prove with a similar argument as for $C^l_K(\timef(M),H^s_\loc(E\vert_{\Sigma_{\bullet}}))$ that the embedding
\begin{equation*}
L^2_{\loc,K}(\timef(M),H^s_\loc(E\vert_{\Sigma_{\bullet}}))\hookrightarrow C^{-\infty}(M,E)
\end{equation*}
is continuous. In order to topologise this space, one introduces the seminorms
\begin{equation*}
\Vert u \Vert^2_{I,K,s}:=\int_{I} \Vert u \Vert^2_{H^s_\loc(E\vert_{\Sigma_t})} \differ t
\end{equation*}
for any compact subinterval $I$ in $\timef(M)$. This turns $L^2_{\loc,K}(\timef(M),H^s_\loc(E\vert_{\Sigma_{\bullet}}))$ into a Fr\'{e}chet space from which one can define the LF-space
\begin{equation*}
L^2_{\loc,\scomp}(\timef(M),H^s_\loc(E\vert_{\Sigma_{\bullet}}))=\bigcup_{\substack{K \subset M \\ K\,\,\text{spatially} \\ \text{compact}}} L^2_{\loc,K}(\timef(M),H^s_\loc(E\vert_{\Sigma_{\bullet}})) \quad.
\end{equation*}
A (dense) subset of $L^2_{\loc,\scomp}(\timef(M),H^s_\loc(E\vert_{\Sigma_{\bullet}}))$ is the space of smooth sections of $E$ with spatially compact support on $M$:
\begin{equation*}
C^{\infty}_{\scomp}(M,E)=\bigcup_{\substack{K \subset M \\ K\,\,\text{spatially} \\ \text{compact}}} C^\infty_{K}(M,E) \quad.
\end{equation*}
%
\section{Spin structures and Dirac operators on globally hyperbolic spacetimes}\label{chap:Dirac}

We first repeat some basics about a spin structure and Dirac operators on pseudo-Riemannian manifolds. Afterwards, we consider the decomposition of the Dirac operator along a Riemannian hypersurface in a time-oriented Lorentzian manifold. The content of this section relies on the detailed explainations in the textbooks \cite{baum1981spin} and \cite{LawMi} with supporting notes from \cite{BaerGauMor}, \cite{vdDungen} and \cite{Gin}. 

\subsection{General aspects}\label{chap:Dirac-sec:Spinstruc}
%
We turn our attention on a time- and space-oriented $n$-dimensional ($n$ even) pseudo-Riemannian spin manifold $M$ with metric $\met$ of signature $(r,s)$. Such a manifold comes with a \textit{spin(or) bundle} $\spinb(M)$ which is defined as complex vector bundle, associated to the spinor representation of the representation space in $\R^n$. We call sections of $\spinb(M)$ \textit{spinor fields} and recall that for even dimensions the spinor bundle decomposes into the two submodules $\spinb^{\pm}(M)$: 
\begin{equation}\label{chiraldecompbund}
\spinb(M)=\spinb^{+}(M)\oplus\spinb^{-}(M)\quad.
\end{equation}
These subbundles are called \textit{positive} respectively \textit{negative half-spin(or) bundles} and we call their sections spinor fields of positive respectively negative chirality. A Clifford multiplication operation is given by a pointwise acting vector space homomorphism $\mathbf{c}$ from $T_pM$ to $\End(\spinb_p(M))$, satisfying
\begin{equation}\label{cliffmult}
\cliff{X}\cdot\cliff{Y}+\cliff{Y}\cdot\cliff{X}=-2\met_p(X,Y)\Iop{\spinb_p(M)}\quad \forall\,X,Y \in T_pM \quad.
\end{equation}
In addition, there exists a non-degenerate Hermitian bundle product $\idscal{1}{\spinb(M)}{\cdot}{\cdot}$ on the spinor bundle which is only positive definite in the Riemannian case. The action of $\cliff{X}$ on a spinor becomes formally self-adjoint:
\begin{equation}\label{cliffmult2}
\idscal{1}{\spinb_p(M)}{\cliff{X}u}{v}+(-1)^s\idscal{1}{\spinb_p(M)}{u}{\cliff{X}v}=0\quad \forall\, X \in T_pM 
\end{equation}
for $u,v\in \spinb_p(M)$. A consequence of \clef{cliffmult} and \clef{cliffmult2} is the following scaling property:
\begin{equation}\label{almostunitary}
\idscal{1}{\spinb_p(M)}{\cliff{X}u}{\cliff{X}v}=(-1)^s\met(X,X)\idscal{1}{\spinb_p(M)}{u}{v} \,.
\end{equation}
Integrating $\idscal{1}{\spinb_p(M)}{\cdot}{\cdot}$ over the manifold against the volume form $\dvol{\met}$ gives a sesquilinear form $\idscal{1}{C^\infty_\comp(\spinb(M))}{\cdot}{\cdot}$ which is positive definite in the Riemannian case, but in general only non-degenerate. In order to construct $L^2$-spaces for spinor sections as Hilbert spaces, we have to consider the completion of $C^\infty_\comp(\spinb(M))$ with respect to $\idscal{1}{C^\infty_\comp(\spinb(M))}{\cdot}{\cdot}$ which only works for $M$ being a Riemannian manifold. However, one is able to relate the non-degenerate bundle metric to a positive definite Hermitian sesquilinear form. For Lorentzian manifolds, which is the case of our interest, this works as follows: if $M$ is space- and time-oriented, it admits a global unit and timelike vector field $\mathfrak{t}$ such that
\begin{equation}\label{spininnerprodpos}
\dscal{1}{\spinb(M)}{\cdot}{\cdot}:=\idscal{1}{\spinb(M)}{\cliff{\mathfrak{t}}\cdot}{\cdot}
\end{equation}
is a positive definite bundle metric on $\spinb(M)$ which then implies $\idscal{1}{C^\infty_\comp(\spinb(M))}{\cdot}{\cdot}$ to be an inner product. The completion of $C^\infty_\comp(\spinb(M))$ with respect to this inner product then induces a Krein space where $\cliff{\mathfrak{t}}$ acts as fundamental symmetry, since $\cliff{\mathfrak{t}}$ is self-adjoint by \clef{cliffmult2} ($s=1$) and unitary from \clef{almostunitary} because $\met(\mathfrak{t},\mathfrak{t})=-1$ holds:
\begin{equation}\label{cliffisom} 
\idscal{1}{\spinb(M)}{\cliff{\mathfrak{t}}u}{\cliff{\mathfrak{t}}v}=-\met(\mathfrak{t},\mathfrak{t})\idscal{1}{\spinb(M)}{u}{v}=\idscal{1}{\spinb(M)}{u}{v}\quad.
\end{equation}
In order to define the spin- or Atiyah-Singer Dirac operator, we have to clarify how to implement a connection on $\spinb(M)$. It is proven in \cite[Satz 3.2]{baum1981spin} that the local expression for the spin Levi-Civita connection $\Nabla{\spinb(M)}{}$ is related to the Levi-Civita connection $\nabla$ of $TM$: 
\begin{equation}\label{spinLCconnection}
\Nabla{\spinb(M)}{X}\Psi=\left[A,X(\Psi)+\frac{1}{2}\sum_{k<l}\epsilon_k\epsilon_l\met(\nabla_X e_k,e_l)\cliff{e_k}\cdot\cliff{e_l}\psi\right]
\end{equation}
where $A$ is an element of the spin frame bundle and $\psi$ an element of the spin representation space. The action of $X$ on the local function $\psi$ with values in the representation space is defined as for scalar functions. Replacing $X$ with an element $e_j$ in the tangent frame allows to rewrite $\met(\nabla_{e_j} e_k,e_l)$ with Christoffel symbols. $\Nabla{\spinb(M)}{}$ complies with the Leibniz rule which is inherited from the one of $\nabla$. One can further determine the connection in such a way that the following compatability conditions are additionally satisfied: 
\begin{eqnarray}
X\idscal{1}{\spinb(M)}{u}{v}&=&\idscal{1}{\spinb(M)}{\Nabla{\spinb(M)}{X}u}{v}+\idscal{1}{\spinb(M)}{u}{\Nabla{\spinb(M)}{X}v} \label{compspin1}\\
\Nabla{\spinb(M)}{X}(\cliff{Y}u)&=&\cliff{\nabla_{X}Y}u+\cliff{Y}\Nabla{\spinb(M)}{X}u \label{compspin2}
\end{eqnarray}
where $u,v\in C^\infty(\spinb(M))$ and $X,Y \in C^\infty(TM)$. Condition \clef{compspin2} implies that the Clifford multiplication homomorphism and thus the spinorial volume form become globally parallel. As $\spinb^{\pm}(M)$ can be viewed as eigenbundles of the spinorial volume form, \clef{compspin2} implies that the spin connection preserves the eigenspace decomposition \clef{chiraldecompbund}. \\
\\
The \textit{Dirac operator} is the map $\Dirac\,:\,C^\infty(\spinb(M))\,\rightarrow\,C^\infty(\spinb(M))$ which is locally given via
\begin{equation}\label{diracM}
\Dirac \Psi := \sum_{j=1}^n \epsilon_{j}\cliff{e_j} \Nabla{\spinb(M)}{e_j} \Psi \\
\end{equation}
for a local orthonormal pseudo-Riemannian tangent frame $e_1,...,e_n$. Its principal symbol is
\begin{equation*}
{\sigma}_1(\Dirac)(x,\xi)\Psi=\Imag\cliff{\xi^\sharp}\Psi(x)\, ;
\end{equation*}
the sharp isomorphism is taken with respect to $\met$. It shows that $\Dirac$ is not-elliptic if $s\neq 0$ as it can vanish for a lightlike covector. \clef{diracM} respects the splitting \clef{chiraldecompbund}, but maps sections of the subbundle $\spinb^{+}(M)$ to $\spinb^{-}(M)$ and vice versa. Hence, $\Dirac$ can be represented as
\begin{equation}\label{directsumrep}
\Dirac=\left(\begin{matrix} 0 & D_{-} \\ D_{+} & 0 \end{matrix}\right)\quad\text{with}\quad D_{\pm}\in \Diff{1}{}(\spinb^{\pm}(M),\spinb^{\mp}(M))\quad.
\end{equation}
Both $D_{\pm}$ are also non-elliptic as long as $0<s<n$. The \textit{Dirac-Laplacian} $\Dirac^2$ is a differential operator of second order with principal symbol ${\sigma}_2(\Dirac^2)(x,\xi)=\met(\xi^\sharp,\xi^\sharp)\Iop{\spinb(M)}$ which is also a non-elliptic operator apart from the Riemannian case. In particular, this shows that $\Dirac^2$ and thus $D_{\pm}D_{\mp}$ are a normally hyperbolic operators. The difference between the Dirac-Laplacian and the Bochner-Laplacian $\Nabla{\spinb(M)}{}^\ast\Nabla{\spinb(M)}{}$ is expressed in terms of curvature quantities and known as (pseudo-Riemannian) Lichnerowicz-Weitzenböck formula.\\
\\
We now specify $M$ to be a $(n+1)$-dimensional globally hyperbolic spin manifold with $n$ odd. We consider any compact, but fixed time interval $[t_1,t_2]$ such that $M\vert_{[t_1,t_2]}$ becomes temporal compact. The global hyperbolicity implies that $M\vert_{[t_1,t_2]}$ becomes a manifold with boundary $\Sigma_1\sqcup\Sigma_2$ with $\Sigma_i=\Sigma_{t_i}$ for $i \in \SET{1,2}$. The defect of $\Dirac$ to be formally skew-adjoint can be expressed in terms of boundary integrals.
\newpage
\begin{prop}\label{diracselfadprop}
Suppose $M$ is a globally hyperbolic spin manifold with spacelike Cauchy hypersurface $\Sigma$, Lorentzian metric \clef{orthmet} and $\Dirac$ the Dirac operator; for any time interval $[t_1,t_2]$ and spinor fields $u,v \in C^\infty_\scomp(\spinb(M))$ with temporal support in this time interval we have
\begin{multline}\label{diracselfad}
\int_M \idscal{1}{\spinb(M)}{\Dirac u}{v}+\idscal{1}{\spinb(M)}{u}{\Dirac v} \dvol{} \\= \int_{\Sigma_{2}} \idscal{1}{\spinb(M)}{u}{\cliff{\mathsf{v}}v} \dvol{\Sigma_t}-\int_{\Sigma_{1}} \idscal{1}{\spinb(M)}{u}{\cliff{\mathsf{v}}v} \dvol{\Sigma_t} \, .
\end{multline}
\end{prop}
\begin{proof}
We choose any, but fixed time interval $[t_1,t_2]$ and two spinors $u,v$ as in the claim. Recalling the preface of this claim, it is enough to consider $M\vert_{[t_1,t_2]}$ due to the temporal support of the spinors. The boundary of $M\vert_{[t_1,t_2]}$ is given by $\Sigma_{1}\sqcup \Sigma_{2}$. Adapting the steps of the proof in \cite[Prop. 5.3]{LawMi} from the Riemannian to the Lorentzian setting gives 
\begin{equation}\label{lemstep3}
\idscal{1}{\spinb(M)}{\Dirac u}{v}+\idscal{1}{\spinb(M)}{u}{\Dirac v}=\sum_{j=0}^n\epsilon_j \Div\left(\idscal{1}{\spinb(M)}{u}{\cliff{e_j}v} e_j \right)\quad.
\end{equation}
Because $M$ is time and space-oriented, there is a global unit timelike vector $\mathsf{v}$ which we choose to be past-directed. We choose the future-oriented Lorentz-orthonormal tangent frame in such a way that $e_0$ is future-timelike, indicating $e_0=-\mathsf{v}$ and $\epsilon_0=-1$; the other members in the frame are spacelike with $\epsilon_j=1$ for $j>0$. We apply the divergence theorem for Lorentzian manifolds 
with a timelike unit normal vector $\mathfrak{n}$ where $\mathfrak{n}=-\mathsf{v}$ is inwards-pointing on $\Sigma_{1}$ and coincides with $\mathsf{v}$ on $\Sigma_{2}$ to assure that $\mathfrak{n}$ is also inwards-pointing there. Integrating the right-hand side of \clef{lemstep3} finally leads to \clef{diracselfad}.\qedhere

\end{proof}
Formally skew-adjointness of $\Dirac$ follows for spinor fields with compact support in the interior of $M$: $\Dirac^\dagger=-\Dirac$. \clef{directsumrep} then implies $D^\dagger_{\pm}=-D_{\mp}$.\\
\\
Given a Hermitian vector bundle $E\rightarrow M$ with metric connection $\nabla^{E}$. It is conceivable to twist the the spinor bundle with $E$ such that sections of this twisted bundle take values in $E$. We denote the twisted spinor bundle with $\spinb_{E}(M):=\spinb(M)\otimes E$. The subbundles $\spinb^{\pm}_E(M)$ are defined via chiral decomposition of $\spinb_E(M)$. Clifford multiplication becomes $\cliff{\cdot}\otimes\Iop{E}$ and the spinorial covariant derivative becomes a tensor product covariant derivative $\nabla^{\spinb_{E}(M)}$. The \textit{twisted Dirac operator} $\Dirac^E$ is then defined like in \clef{diracM} but with these two modifications. We denote the twisted Dirac operators, which act on the subbundles, with $D^E_{\pm}$. The Lichnerowicz-Weitzenböck formula then becomes 
\begin{equation}\label{lichtwistspinpseudo}
(\Dirac^E)^2=\Nabla{\spinb_E(M)}{}^\ast\Nabla{\spinb_E(M)}{}+\frac{1}{4}\curvcon + \mathfrak{R}^{E}
\end{equation}
where $\curvcon$ is the scalar curvature and $\mathfrak{R}^{E}$ the curvature endomorphism
\begin{equation}\label{curvatureendo}
\mathfrak{R}^{E}(\psi \otimes f):= \sum_{i,j=1}^n \left(\cliff{e_i} \cliff{e_j}\psi\otimes \left(\mathcal{R}^{E}_{e_i,e_j}f\right)\right)
\end{equation}
for $f\in C^\infty(M,E)$ and $\psi\in C^\infty(\spinb(M))$; $\mathcal{R}^{E}$ is the bundle curvature of $E$. 
\begin{rem}\label{remsdiracop}
\begin{itemize}
\item[]
\item[(i)] \Cref{diracselfadprop} remains true if we merely consider differentiable spinor fields with the same support properties and if we replace $\Dirac$ with its twisting version $\Dirac^{E}$. 
\item[(ii)] Recalling the inner product \clef{spininnerprodpos}, we observe that we can define a bundle metric which is a positive definite and Hermitian. We set $\upbeta:=\cliff{\mathsf{v}}$ for the global unit timelike vector field; the bundle product 
\begin{equation}\label{spininnerprodposglobhyp}
\dscal{1}{\spinb(M)}{\cdot}{\cdot}:=\idscal{1}{\spinb(M)}{\upbeta\cdot}{\cdot} 
\end{equation}
is positive definite and induces an inner product for $\phi,\psi \in C^\infty_\comp(\spinb(M))$:
\begin{equation}\label{spinl2innerprodpos}
\dscal{1}{L^2(\spinb(M))}{\phi}{\psi}:=\int_M \dscal{1}{\spinb(M)}{\phi}{\psi} \dvol{\met}=\int_M \idscal{1}{\spinb(M)}{\upbeta \phi}{\psi} \dvol{\met} \quad.
\end{equation}
The completion of $C^\infty_\comp(\spinb(M))$ with respect to the norm, induced by \clef{spinl2innerprodpos}, defines the space of square-integrable spinor fields $L^2(\spinb_E(M))$. We
get with \clef{cliffisom} the isometry property:
\begin{equation}\label{isombetatwist1}
\dscal{1}{\spinb(M)}{\upbeta \phi}{\upbeta\psi}=\dscal{1}{\spinb(M)}{\phi}{\psi} \quad .
\end{equation} 
This carries over to an isometry on $L^2(\spinb(M))$. These facts won't change if we consider twisted spinor fields. We only need to replace $\upbeta$ with $\upbeta\otimes\Iop{E}$, leading to
\begin{equation}\label{isombetatwist12}
\dscal{1}{\spinb_E(M)}{(\upbeta\otimes\Iop{E}) \Phi}{(\upbeta\otimes\Iop{E})\Psi}=\dscal{1}{\spinb_E(M)}{\Phi}{\Psi}
\end{equation} 
for $\Phi,\Psi \in C^\infty(\spinb_E(M))$ such that we also get an isometry on $L^2(\spinb_E(M))$. 
\end{itemize}
\end{rem}

\subsection{Spin structure and Dirac operator along a Cauchy hypersurface}\label{chap:Dirac-sec:spinhyp}

If $M$ is an even dimensional globally hyperbolic spin manifold with spacelike Cauchy hypersurfaces $\Sigma$ (of odd dimension $n$), there exists a global unit timelike vector field $\mathsf{v}$ which we choose to be past-directed and normal to $\Sigma_t$ at each $t\in \timef(M)$. Recalling \Cref{remsdiracop} (ii), the vector field $\mathsf{v}$ induces an isometry according to \clef{isombetatwist1}. \clef{cliffmult} furthermore implies that $\upbeta^2=\Iop{\spinb(M)}$ as well as the anti-commuting with $\mathbf{c}$ and it induces the Hermitian bundle metric \clef{spininnerprodposglobhyp} and thus the inner product \clef{spinl2innerprodpos}. The spin structures of the hypersurfaces $\Sigma$ and $\SET{\Sigma_t}_{t\in \timef(M)}$ are inherited from the one on $M$. The existence and implementation of a spin structure on the boundary hypersurface $\Sigma$ is explained in \cite[Chap.5]{BaerGauMor} or in \cite[Sec.3.1]{vdDungen} with some details. Hence, we obtain a spinor bundle on each $\Sigma_t$:
$$\spinb(M)\vert_{\Sigma_t}=\spinb^{+}(M)\vert_{\Sigma_t}\oplus \spinb^{-}(M)\vert_{\Sigma_t}=\spinb(\Sigma_t)\oplus \spinb(\Sigma_t) $$
for all $t\in \timef(M)$. As the hypersurfaces are odd-dimensional, there is no decomposition into subbundles of different chiralities. We still write $\spinb^{\pm}(\Sigma_t)$ to stress which spinor eigenbundle on $M$ is restricted to the hypersurface. According to the mentioned references, one can define Clifford multiplication for spinors on the hypersurfaces with a vector field $X$ on $\Sigma_t$ by setting
\begin{equation}\label{restcliffglobhyp}
\Cliff{t}{X}:= \Imag \upbeta \cliff{X} \quad;
\end{equation} 
if it acts on sections of $\spinb^{\pm}(M)\vert_{\Sigma_t}$, it acts as $(\pm\Cliff{t}{X})$. A Clifford relation on the hypersurface is inherited from the one on $M$: 
\begin{equation}\label{clifftstruc}
\Cliff{t}{X}\Cliff{t}{Y}+\Cliff{t}{Y}\Cliff{t}{X}=-2\met_t(X,Y)\Iop{\spinb(\Sigma_t)}
\end{equation}
for vector fields $X,Y$ along $\Sigma_t$. All of this carries over to the twisted case where the restricted Clifford multiplication then takes the form $\Cliff{t}{\cdot}\otimes\Iop{E}\vert_{\Sigma_t}$. The induced Hermitian bundle metric \clef{spininnerprodposglobhyp} in \Cref{remsdiracop} (ii) induces an $L^2$-inner product
\begin{equation}\label{spinl2innerprodposrest}
\dscal{1}{L^2(\spinb(\Sigma_t))}{u}{v}:=\int_{\Sigma_t} \dscal{1}{\spinb(\Sigma_t)}{u}{v} \dvol{\met_t}=\int_{\Sigma_t} \idscal{1}{\spinb(\Sigma_t)}{\upbeta u}{v} \dvol{\met_t}
\end{equation}
for $u,v$ spinor fields along the hypersurface. $\upbeta$ still acts as an isometry with respect to the induced bundle metric such that
\begin{equation}\label{isombeta}
\dscal{1}{\spinb(\Sigma_t)}{\upbeta \phi}{\upbeta \psi}=\dscal{1}{\spinb(\Sigma_t)}{\phi}{\psi} \quad;
\end{equation} 
its pendent for the twisted case is 
\begin{equation}\label{isombetatwist}
\dscal{1}{\spinb_{E}(\Sigma_t)}{(\upbeta\otimes\Iop{E}) \Phi}{(\upbeta\otimes\Iop{E}) \Psi}=\dscal{1}{\spinb_{E}(\Sigma_t)}{\Phi}{\Psi} \quad .
\end{equation} 
We observe that $\mathbf{c}_t$ is formally skew-adjoint for all $t$: 
\begin{equation}\label{clifftskew}
\dscal{1}{\spinb(\Sigma_t)}{\Cliff{t}{X}u}{v}=-\dscal{1}{\spinb(\Sigma_t)}{u}{\Cliff{t}{X}v} \quad.
\end{equation}
This skew-adjointness carries over to the inner product \clef{spinl2innerprodposrest}.\\
\\
For a Lorentz-orthonormal tangent frame $e_0=\mathsf{v},e_1,...,e_n$, the subframe $e_1,...,e_n$ becomes a Riemann-orthonormal tangent frame for $\Sigma_t$. We first manipulate the sum in the local expression \clef{spinLCconnection} with \clef{chrissymb} and get
\begin{equation*}
\frac{1}{2}\sum_{k<l}\Gamma_{jk,l}\epsilon_k\epsilon_l\cliff{e_k}\cdot\cliff{e_l}=\frac{1}{2}\sum_{1\leq k<l\leq n}\Gamma^{\Sigma_t}_{jk,l}\cliff{e_k}\cdot\cliff{e_l}+\frac{1}{2}\upbeta\cdot\cliff{\wein(e_j)}
\end{equation*}
and thus for a section $u$ of $\spinb(M)\vert_{\Sigma_t}$
\begin{equation}\label{spincovsplit}
\left.\Nabla{\spinb(M)}{X}u\right\vert_{\Sigma_t}=\left.\Nabla{\spinb(\Sigma_t)}{X}u\right\vert_{\Sigma_t}+\left.\frac{1}{2}\upbeta\cliff{\wein(X)}u\right\vert_{\Sigma_t}\quad.
\end{equation}
This enables us to decompose the Dirac operator along a fixed hypersurface with \Cref{lem2-1} (2): 
\begin{equation*}
\left.\Dirac u \right\vert_{\Sigma_t}=\left.-\upbeta\Nabla{\spinb(\Sigma_t)}{\mathsf{v}}u\right\vert_{\Sigma_t}+\sum_{j=1}^n \left.\cliff{e_j} \Nabla{\spinb(\Sigma_t)}{e_j}u\right\vert_{\Sigma_t}
+\left.\frac{1}{2} \sum_{j=1}^n \cliff{e_j}\upbeta\cliff{\wein(e_j)}u\right\vert_{\Sigma_t} \quad.
\end{equation*}
The remaining triple Clifford multiplication follows from calculating the commutator  
\begin{equation}\label{cliffrelhelp}
-[\upbeta \cliff{e_j},\cliff{\wein(e_j)}]=2\cliff{e_j}\upbeta\cliff{\wein(e_j)}=2\met_t(e_j,\wein(e_j))\upbeta 
\end{equation}
such that the final expression for the decomposition becomes
\begin{equation*}
\left. \Dirac u \right\vert_{\Sigma_t} =-\left(\left.\upbeta\Nabla{\spinb(\Sigma_t)}{\mathsf{v}}u\right\vert_{\Sigma_t}\pm\Imag\upbeta \sum_{j=1}^n \left.\Cliff{t}{e_j} \Nabla{\spinb(\Sigma_t)}{e_j}u\right\vert_{\Sigma_t} -\left.\frac{1}{2} \upbeta \tr{\met_t}{\wein}u\right\vert_{\Sigma_t}\right)\quad .
\end{equation*}
With $nH_t=\tr{\met_t}{\wein}$ as mean curvature of the hypersurface $\Sigma_t$ and 
\begin{equation}\label{Dirachyp}
\mathcal{A}_t:=\left(\begin{matrix}
A_t & 0 \\
0 & -A_t
\end{matrix} \right)\quad \text{with} \quad A_t=\sum_{j=1}^n \Cliff{t}{e_j} \Nabla{\spinb(\Sigma_t)}{e_j} 
\end{equation}
as the \textit{hypersurface Dirac operator} on sections of $\spinb(M)\vert_{\Sigma_t}$ for an odd dimensional submanifold, the Dirac operator finally becomes
\begin{equation}\label{diracMsubdecomp}
\left. \Dirac u \right\vert_{\Sigma_t}=-\left.\upbeta \left(\Nabla{\spinb(\Sigma_t)}{\mathsf{v}}\pm\Imag \mathcal{A}_t -\frac{n}{2}H_t \right)u\right\vert_{\Sigma_t} 
\end{equation}
and \clef{directsumrep} implies
\begin{equation} \label{dirachyppos}
\left. D_{\pm} u \right\vert_{\Sigma_t} = -\upbeta \left.\left(\Nabla{\spinb(\Sigma_t)}{\mathsf{v}}\pm\Imag A_t  -\frac{n}{2}H_t \right)u\right\vert_{\Sigma_t}=-\upbeta \left.\left(\Nabla{\spinb(\Sigma_t)}{\mathsf{v}}+ B_{t,\pm} \right)u\right\vert_{\Sigma_t} \quad
\end{equation}
where we introduced the abbreviation $B_{t,\pm}:=\pm\Imag A_t - \frac{n}{2} H_t$ which is an operator of most first order, acting tangential to the hypersurface. If we further twist the Dirac operator with a Hermitian vector bundle $E$, \clef{Dirachyp} results in the \textit{twisted hypersurface Dirac operator}
\begin{equation}\label{Dirachyptwist}
\mathcal{A}^E_t:=\left(\begin{matrix}
A^E_t & 0 \\
0 & -A^E_t
\end{matrix} \right)\quad \text{with} \quad A^E_t:=\sum_{j=1}^n (\Cliff{t}{e_j}\otimes\Iop{E}) \Nabla{\spinb_E(\Sigma_t)}{e_j} \quad.
\end{equation}
Hence, we gain for a section $u$ of $\spinb^{\pm}_{E}(M)\vert_{\Sigma_t}$
\begin{equation}\label{diracMsubdecomptwist}
\left. \Dirac^E u \right\vert_{\Sigma_t}=-\left.(\upbeta\otimes\Iop{E}) \left(\Nabla{\spinb_E(\Sigma_t)}{\mathsf{v}}\pm\Imag \mathcal{A}^E_t -\frac{n}{2}H_t\Iop{\spinb_E(\Sigma_t)} \right)u\right\vert_{\Sigma_t} 
\end{equation}
and consequently
\begin{equation} \label{dirachyppostwist}
\begin{split}
\left. D^E_{\pm} u \right\vert_{\Sigma_t} &= -(\upbeta\otimes\Iop{E})\left.\left(\Nabla{\spinb_E(\Sigma_t)}{\mathsf{v}}+ B^E_{t,\pm} \right)u\right\vert_{\Sigma_t} 
\end{split}
\end{equation}
with $B^E_{t,\pm}:=\pm\Imag A^E_t - \frac{n}{2} H_t\Iop{\spinb_E(\Sigma_t)}$.\\
\\
We now consider in contrast that $M$ is a Riemannian topological cylinder with base $\Sigma$ and Riemannian metric $\check{\met}$, which we mark with $\check{M}$. $\mathsf{v}$ is now a global normal vector field at each $\Sigma_t$. If $\check{M}$ is spin, it carries a spinor bundle $\spinb(\check{M})$ which also splits into two subbundles $\spinb^{\pm}(\check{M})$ for even dimensions. As the Clifford multiplication homomorphism depends on the metric, we denote it with $\ccliff{\cdot}$ and in particular we set $\check{\upbeta}:=\ccliff{\mathsf{v}}$ which still acts as isometry with $\check{\upbeta}^2=-\Iop{\spinb(\check{M})}$. The Riemannian Dirac operator $\check{\Dirac}$ is then defined as in \clef{diracM} without the signature elements and $\check{D}_{\pm}$ due to the chirality splitting. The restriction of the spinor bundle to the hypersurfaces $\Sigma_t$ coincide with the restricted spinor bundles in the Lorentzian case: $\spinb^{\pm}(\Sigma_t)=\spinb^{\pm}(\check{M})\vert_{\Sigma_t}$. The restricted Clifford multiplication homomorphism $\cCliff{t}{\cdot}:= \ccliff{\cdot}\check{\upbeta}$ is still skew-adjoint and inherits a Clifford algebra for spinor fields along the hypersurfaces. The decompositions of the twisted Riemannian Dirac operators $\check{D}^E$ and $\check{D}^E_{\pm}$ into tangential and normal components relative to $\Sigma_t$ works out similarly with the help of the Christoffel symbols in \clef{chrissymb}, leading for a section $u\in C^\infty(\spinb^{\pm}_{E}(\check{M}))$ to 
\begin{equation}\label{riemspincovsplit}
\left.\Nabla{\spinb_E(\check{M})}{X}u\right\vert_{\Sigma_t}=\left.\Nabla{\spinb_E(\Sigma_t)}{X}u\right\vert_{\Sigma_t}-\left.\frac{1}{2}(\check{\upbeta}\otimes\Iop{E})(\ccliff{\wein(X)}\otimes\Iop{E})u\right\vert_{\Sigma_t}\quad,
\end{equation}
implying
\begin{equation}\label{diracMsubdecomptwistRiem}
\left. \check{\Dirac}^E u \right\vert_{\Sigma_t}=-(\check{\upbeta}\otimes\Iop{E}) \left.\left(-\Nabla{\spinb_E(\Sigma_t)}{\mathsf{v}}\pm \mathcal{A}^E_t -\frac{n}{2}H_t\Iop{\spinb_E(\Sigma_t)} \right)u\right\vert_{\Sigma_t} 
\end{equation}
and
\begin{equation} \label{dirachyppostwistRiem}
\left. \check{D}^E_{\pm} u \right\vert_{\Sigma_t} = -(\check{\upbeta}\otimes\Iop{E}) \left.\left(-\Nabla{\spinb_E(\Sigma_t)}{\mathsf{v}}\pm A^E_t -\frac{n}{2}H_t\Iop{\spinb_E(\Sigma_t)} \right)u\right\vert_{\Sigma_t} \quad .
\end{equation}

\section{Well-posedness of the Cauchy problem for the Dirac equation}\label{chap:Cauchy}

Given the Dirac operators $D^E_{\pm}$ for twisted spinor fields on a globally hyperbolic spin manifold with non-compact, but complete Cauchy hypersurface $\Sigma$; these are from now on our key assumptions. The main interest is focused on the Cauchy problem
\begin{equation}\label{Diraceq}
D_{\pm}^{E} u = f\quad\text{with}\quad u\vert_{\Sigma_t}=g
\end{equation}
for a suitable weak solution $u$ as section of $\spinb_{E}^{\pm}(M)$, $f$ a suitable section of $\spinb_{E}^{\mp}(M)$ and $g$ a Sobolev section of $\spinb_{E}^{\pm}(\Sigma_t)$ which we will specify in the following subsections. We start with some preparatory results in order to prove and calculate an energy estimate which becomes crucial in showing uniqueness and thus the main result of this section respectively this paper.  

\subsection{Some preparatory results}\label{chap:cauchy-sec:wellpos-1}

Several special relations turn out to be useful in computing energy estimates. We start with the untwisted case. 
\begin{lem} \label{lemenest1} 
The following relations hold for a vector field $X\in C^\infty(T\Sigma_t)$ and a smooth section $u$ of $\spinb(M)$ along a hypersurface $\Sigma_t$ for each $t\in \timef(M)$: 
\begin{itemize}
\item[(1)] $\Nabla{\spinb(\Sigma_t)}{X}(\upbeta u)=\upbeta \Nabla{\spinb(\Sigma_t)}{X}u$ and $A_t (\upbeta u)=-\upbeta A_t u$ ;
\end{itemize}
if moreover each $\Sigma_t$ is complete, then
\begin{itemize}
\item[(2)] $\Lambda_{t}^s\upbeta=\upbeta\Lambda_{t}^s$ for $s\in \R$ where $\Lambda_{t}^2=\Iop{}+(\Nabla{\spinb(\Sigma_t)}{})^\ast\Nabla{\spinb(\Sigma_t)}{}$;
\item[(3)] $\dscal{1}{L^2(\spinb(\Sigma_t))}{B_{t,\pm} v}{w}+\dscal{1}{L^2(\spinb(\Sigma_t))}{v}{B_{t,\pm} w}=-n H_t\dscal{1}{L^2(\spinb(\Sigma_t))}{v}{w}$ for all $v,w\in C^\infty(\spinb^{\pm}(\Sigma_t))$, sharing the same chirality.
\end{itemize}
\end{lem}
\begin{proof}
\begin{itemize} \item[]
\item[(1)] The first commutativity is a consequence of the compatibility with Clifford multiplication \clef{compspin2} and $\upbeta \Cliff{t}{X}=-\Cliff{t}{X}\upbeta$ for all $t$, implying the anti-commuting with $A_t$.
\item[(2)] Denote with $\curvcon_{\Sigma_t}$ the (Ricci) scalar curvature for $\Sigma_t$. The Lichnerowicz formula for the hypersurface Dirac operator and result (1) lead to
\begin{equation}\label{midestep1wellpos1}
(\Nabla{\spinb(\Sigma_t)}{})^\ast\Nabla{\spinb(\Sigma_t)}{}(\upbeta u)= \upbeta (\Nabla{\spinb(\Sigma_t)}{})^\ast\Nabla{\spinb(\Sigma_t)}{}u
\end{equation}
and thus $\Lambda_{t}^2(\upbeta u)= \upbeta \Lambda_{t}^2 u$. This holds true for any positive even power $\Lambda_{t}^{2k}$, $k \in \N_0$, after applying \clef{midestep1wellpos1} $k$ times and thus for any polynomial in $\Lambda_{t}^{2}$. As $\Lambda_{t}^2$ is essentially self-adjoint on $L^2(\spinb^{\pm}(\Sigma_t))$ with positive spectrum by completeness of the hypersurfaces, the function $f(x)=x^{s/2}$ is continuous on the spectrum of $\Lambda_{t}^2$. Consequently, $\Lambda_{t}^s=f(\Lambda_{t}^2)$ is defined via the limit of any sequence of polynomials in $\Lambda_t^2$, converging uniformly on the spectrum to $f$. $\Lambda_{t}^s(\upbeta u)= \upbeta \Lambda_{t}^s u$ then follows for any $s\in \R$ because the action of $\upbeta$ commutes with the limit and with each element of the sequence.   
\item[(3)] The density of $C^\infty_\comp$ in $L^2$ allows to restrict the proof to smooth and compactly supported spinor fields. The left-hand side of the equation in the claim gives for both chiralities
\begin{equation*}
\pm\left(\dscal{1}{\spinb(\Sigma_t)}{\Imag A_t v}{w}+\dscal{1}{\spinb(\Sigma_t)}{v}{\Imag A_t w}\right)-n H_t\dscal{1}{\spinb(\Sigma_t)}{v}{w} \quad.
\end{equation*}
$A_t$ is formally self-adjoint with respect to the induced inner product on $\Sigma_t$ since $\bound \Sigma_t =\emptyset$ for all $t \in \timef(M)$ by assumption and by completeness of the hypersurface; the proof works similarly as the proof of \clef{lemstep3} by using a synchronous Lorentz-orthonormal frame at a point on $\Sigma_t$. Thus, $\Imag A_t$ is formally skew-adjoint with respect to the same inner product and the term in the round brackets vanishes.\\
\\
In order to extend this to a positive definite $L^2$-scalar product, we now use the essential self-adjointness of the Riemannian Dirac operator $A_t$ on $L^2$-spaces since each hypersurface is complete:
\begin{equation}\label{extendB}
\dscal{1}{L^2(\spinb(\Sigma_t))}{B_{t,\pm} v}{w}+\dscal{1}{L^2(\spinb(\Sigma_t))}{v}{B_{t,\pm} w}=-n\dscal{1}{L^2(\spinb(\Sigma_t))}{H_t v}{w} \quad.
\end{equation}
$B_{t,\pm}$ is then defined as in \clef{dirachyppos} with the extension of $A_t$ instead. \qedhere 
\end{itemize}
\end{proof}
This can be reformulated for twisted Dirac operators.
\begin{lem} \label{lemenest1ext} 
The following relations hold for a vector field $X$ and a smooth section $u$ of $\spinb_{E}(M)$ along each complete hypersurface $\Sigma_t$ without boundary for all $t\in \timef(M)$: 
\begin{itemize}
\item[(1)] $\Lambda_{t}^s(\upbeta\otimes \Iop{E})=(\upbeta\otimes \Iop{E})\Lambda_{t}^s$ for $s\in \R$ with $\Lambda_{t}^2=\Iop{}+(\Nabla{\spinb_{E}(\Sigma_t)}{})^\ast\Nabla{\spinb_{E}(\Sigma_t)}{}$;
\item[(2)] $\dscal{1}{L^2(\spinb_{E}(\Sigma_t))}{B_{t,\pm}^{E} u}{v}+\dscal{1}{L^2(\spinb_{E}(\Sigma_t))}{u}{B^{E}_{t,\pm} v}=-n \dscal{1}{L^2(\spinb_{E}(\Sigma_t))}{H_t u}{v}$ for all $u,v\in L^2(\spinb^{\pm}_{E}(\Sigma_t))$ with the same chirality.
\end{itemize}
\end{lem}
The same proof arguments, used for \Cref{lemenest1}, are valid for this modified situation. The twisted Lichnerowicz formula \clef{lichtwistspinpseudo} has to be used for (1) and one shows in addition that $(\upbeta\otimes\Iop{E})$ commutes with the curvature endomorphism \clef{curvatureendo}.\\
\\ 
The following two spaces become crucial for the inhomogeneous and homogeneous solutions of the Cauchy problem in the space of finite energy sections.
\begin{defi}\label{finensolkern}
For any $s \in \R$ the set of \textit{finite} $s$-\textit{energy solutions of} $D^{E}_{\pm}$ is defined by
\begin{multline}\label{finensol}
FE^{s}_{\scomp}(M,\timef,D^{E}_{\pm})\\:=\bigl\{u \in FE^s_{\scomp}(M,\timef,\spinb_{E}^{\pm}(M))\,\big\vert\, D^{E}_{\pm}u \in L^2_{\loc,\scomp}(\timef(M),H^s_\loc(\spinb_{E}^{\mp}(\Sigma_{\bullet})))\bigr\} \, ;
\end{multline}
the set of \textit{finite} $s$-\textit{energy kernel solutions of} $D^{E}_{\pm}$ is 
\begin{equation}\label{finenkern}
FE^{s}_{\scomp}\left(M,\timef,\kernel{D^{E}_{\pm}}\right):=FE^s_{\scomp}(M,\timef,\spinb_{E}^{\pm}(M))\cap \kernel{D^{E}_{\pm}} \, .
\end{equation}
\end{defi}
The kernel solutions come with an interesting property.
\begin{lem}\label{boots}
If $u \in FE^{s}_{\scomp}\left(M,\timef,\kernel{D^{E}_{\pm}}\right)$ and $s > \frac{n}{2}+2$, then $u \in C^1_{\scomp}(\spinb_{E}^{\pm}(M))$.
\end{lem}
\begin{proof}
With \clef{dirachyppostwist} the equation $D^{E}_{+}u=0$ along each hypersurface takes the form
\begin{equation*}
\left. \Nabla{\spinb_{E}(M)}{\partial_t}u \right\vert_{\Sigma_t}=-N\left. \Nabla{\spinb_{E}(M)}{\mathsf{v}}u \right\vert_{\Sigma_t}=N\left(\Imag A^{E}_t - \frac{n}{2}H_t\right)u\vert_{\Sigma_t} \quad.
\end{equation*}
A section $u \in C^0_{\scomp}(\timef(M),H^s_\loc(\spinb_{E}^{+}(\Sigma_{\bullet})))$ has support inside a spatially compact subset in $M$. For $u\vert_{\Sigma_t} \in H^s_\loc(\spinb_{E}^{+}(\Sigma_t))$ at each $t \in \timef(M)$ and $\supp{u}\cap \Sigma_t$ compact, the right-hand side consists of differential operators at most order 1 along each $\Sigma_t$ and $\Nabla{\spinb_{E}(M)}{\partial_t}u\vert_{\Sigma_t}\in H^{s-1}_\loc(\spinb_{E}^{+}(\Sigma_t))$ and consequently $u \in C^1_{\scomp}(\timef(M),H^{s-1}_\loc(\spinb_{E}^{+}(\Sigma_{\bullet})))$. The claim follows by the Sobolev embedding theorem for $s-1 > \frac{n}{2}+1$. The same argumentation holds for $D^E_{-}$. 
\end{proof}
We choose inhomogeneities $f \in L^2_{\loc,\scomp}(\timef(M),H^s_\loc(\spinb_{E}^{\mp}(\Sigma_{\bullet})))$ and $u\vert_{\Sigma_t}\in H^s_\loc(\spinb_{E}^{\pm}(\Sigma_t))$ for any $s\in \R$ as initial data and $t \in \timef(M)$ for each Cauchy problems of $D^{E}_{\pm}$ in \clef{Diraceq}. We start with the stronger condition that $f \in FE^{s-1}_\scomp(M,\timef,\spinb_{E}^{\mp}(M))$; this can be weakened later on, but does not affect the main proof. A time reversal argument is going to be used for the coming energy estimate for which reason a closer look on the time reversed Cauchy problem is needed. The \textit{time reversal map}
\vspace{-0.5cm}
\begin{equation}\label{timerevmap}
\begin{array}{rcccc}
\mathcal{T}&:& M\vert_{[t_1,t_2]} & \rightarrow & M\vert_{[t_1,t_2]} \\[2pt]
& &  (t,x) & \mapsto & ((t_2+t_1)-t,x)
\end{array}
\end{equation}
is smooth and acts as involution since $\mathcal{T}^{\,2}=\Iop{M}$. This makes it a formally self-adjoint diffeomorphism on $M\vert_{[t_1,t_2]}$. We will quote its inverse with the same letter as it is a self-inverse map. The pullback of a spinor field with respect to this diffeomorphism is well-defined as spinor field with respect to a Clifford algebra, generated by the pullback metric $\mathcal{T}^{\ast}\met $:
\begin{equation*}
(\mathcal{T}^{\ast}u)(t,x)=u(\mathcal{T}(t,x))
\end{equation*}
for a smooth spinor field $u$; more details concerning the structure of this scalar like transformation behaviour can be found in \cite{DP}. We use these facts to prove the following statement which provides us a time reversal argument.
\begin{lem}\label{timeinvarinace}
Given a globally hyperbolic manifold $M$ with Cauchy temporal function $\timef$, a spinor bundle $\spinb_{E}(M)$, $K \subset M$ compact and $s \in \R$; the following are equivalent for each time interval $[t_1,t_2] \in \timef(M)$:
\begin{itemize}
\item[(1)] $u \in FE^s_\scomp(M,\timef,\spinb_{E}^{\pm}(M))$ solves the forward time Cauchy problem 
\begin{equation*}
D^{E}_{\pm}u=f \in FE^{s-1}_\scomp(M,\timef,\spinb_{E}^{\mp}(M)) \quad , \quad u\vert_{t}=: u_0 \in H^s_\loc(\spinb_{E}^{\pm}(\Sigma_t))
\end{equation*}
for the Dirac equation at fixed initial time $t \in \timef(M)$.
\item[(2)] $\mathcal{T}^\ast u \in FE^s_\scomp(\mathcal{\mathcal{T}}^{-1}(M),\mathcal{T}(\timef),(\mathcal{T}^{-1})^\ast\spinb_{E}^{\pm}(M))$ solves the backward time Cauchy problem 
\begin{equation*}
(\mathcal{T}^\ast\circ D^{E}_{\pm} \circ \mathcal{T}^\ast)u=\mathcal{T}^\ast f \quad , \quad (\mathcal{T}^\ast u)\vert_{\mathcal{T}(t)}= u_0 \in H^s_\loc(\spinb_{E}^{\pm}(\Sigma_t))
\end{equation*}
with $\mathcal{T}^\ast f \in FE^{s-1}_\scomp(\mathcal{\mathcal{T}}^{-1}(M),\mathcal{T}\circ\timef,(\mathcal{T}^{-1})^\ast(\spinb_{E}^{\mp}(M)))$ for the Dirac equation at fixed initial time $\mathcal{T}(t):=(t_2+t_1-t) \in \timef(M)$.
\end{itemize}
Moreover, $(\mathcal{T}^\ast\circ D^{E}_{\pm} \circ \mathcal{T}^\ast)$ is the twisted Dirac operator for reversed time orientation and takes the form
\begin{equation*}
\left. (\mathcal{T}^\ast \circ D^{E}_{\pm} \circ \mathcal{T}^\ast) v \right\vert_{\Sigma_t} = -\left.(\widetilde{\upbeta}\otimes \Iop{E})\left(\nabla_{\widetilde{\mathsf{v}}}\pm\Imag \widetilde{A}^{E}_t v-\frac{n}{2} \widetilde{H}_t\right) v\right\vert_{\Sigma_t} \quad
\end{equation*}
where $\widetilde{\mathsf{v}}=\mathcal{T}_\ast \mathsf{v}$, $\widetilde{\upbeta}=\cliff{\widetilde{\mathsf{v}}}$, $\widetilde{H}_t$ is the mean curvature with respect to the normal vector $\widetilde{\mathsf{v}}$ and $\widetilde{A}^{E}_t$ the hypersurface Dirac operator, defined as in \clef{Dirachyptwist} with a Riemann-orthonormal tangent frame with respect to $\mathcal{T}^\ast \met_t$. 
\end{lem}
\begin{proof}
We note that the pullback with the time reversion map $\mathcal{T}$ commutes with the tensor product and $\mathcal{T}^{\ast} \Iop{E}\mathcal{T}^\ast=\Iop{E}$ such that we can stick to the untwisted case for the sake of legibility. We choose any, but a fixed time interval $[t_1,t_2]\in \timef(M)$. We assume w.l.o.g that $M$ is temporal compact with respect to this time interval. Both Dirac equations are formally the same if one applies a pullback by $\mathcal{T}$ to both sides and uses the self-inverse property $\mathcal{T}^{\,2}=\Iop{M}$ between $D_{\pm}$ and the spinor $u$ with appropiate chirality. Because finite energy sections are embedded in the set of distributional sections, it is enough to consider smooth regularity for the initial data and inhomogeneity as the statement follows for distributions by duality: given a spin-valued distribution $u \in C^{-\infty}(\spinb(M))$ and a cospinor field $\phi \in C^\infty_\comp(\spinb^\ast(M))$; the action of $D_{\pm}$ respectively $\mathcal{T}^\ast\circ D_{\pm} \circ \mathcal{T}^\ast$ on $u$ under the dual pairing $[\cdot\vert\cdot]$ becomes 
\begin{equation*}
\begin{split}  
\dpair{1}{\spinb(M)}{D_{\pm}u}{\phi}&=-\dpair{1}{\spinb(M)}{u}{D_{\mp}\phi}\\
\text{and}\quad\dpair{1}{\spinb(M)}{(\mathcal{T}^\ast\circ D_{\pm} \circ \mathcal{T}^\ast)u}{\phi}&=-\dpair{1}{\spinb(M)}{u}{(\mathcal{T}^\ast\circ D_{\mp} \circ \mathcal{T}^\ast)\phi}\quad
\end{split}
\end{equation*}
where we have used the formal self-adjointness of $\mathcal{T}$ and $D^\dagger_{\pm}=-D_{\mp}$. The support of $u$ is contained in the future and past light cone, so $\mathcal{T}$ only swaps the time direction wherefore the support still satisfies $\supp{\mathcal{T}^{\ast}u}\subset \Jlight{}(K)$. Suppose $u\vert_{\Sigma_t} \in C^\infty_\comp(\spinb^{\pm}(\Sigma_t))$ and $f \in C^\infty_\comp(\spinb^{\mp}(M))$. \cite[Thm.4]{AndBaer} implies the existence of a unique section $u\in C^\infty(\spinb^{\pm}(M))$ with support $\supp{u}\subset \Jlight{}(K)$ for $K \subset M$ compact, solving $D_{\pm}u=f$ on $M$ with initial condition $u\vert_{\Sigma_t}$. $\mathcal{T}^{\ast}u$ and $\mathcal{T}^{\ast} f$ are defined and smooth; the latter is compactly supported. The initial value has to be imposed at time $\mathcal{T}(t)=(t_2+t_1-t_0-t)$ if $t$ is the time for the initial value for the foreward time Cauchy problem. Hence $(\mathcal{T}^\ast u)\vert_{\Sigma_{\mathcal{T}(t)}}$ coincides with $u_0$ as \clef{timerevmap} is an involution. The same holds true for any initial value with Sobolev regularity since only the metric is influenced by the time reversal, but different metrices lead to equivalent Sobolev norms. The pullback $\mathcal{T}^\ast u$ for a solution $u$ of the forward time Cauchy problem is defined and again smooth with $\supp{\mathcal{T}^\ast u}\subset \Jlight{}(K)$ and solves the backward time Cauchy problem.\\
\\
It is left to check that $\mathcal{T}^{*}D_{\pm}\mathcal{T}^{*}$ along any hypersurface $\Sigma_t$ are also a Dirac operators, given as in \clef{dirachyppos}. The pullback spin-structure is determined by the pullback metric such that the pullback on any Clifford multiplication $\cliff{X}$ with respect to a vector field $X$ becomes the Clifford multiplication with respect to the pushforward $\mathcal{T}_{\ast}X$ at each point:
\begin{equation*}
\mathcal{T}^{\ast}\circ\cliff{X}=\cliff{\mathcal{T}_{\ast}X}\quad .
\end{equation*}
We can apply to each component of \clef{dirachyppos} the pullback on the Dirac operator along any spatial hypersurface by linearity: if a Riemann-orthonormal tangent frame $\SET{e_j}_{j=1}^n$ with respect to $\met_t$ for each $\Sigma_t$ is given, $\SET{\mathcal{T}_\ast e_j}_{j=1}^n=\SET{\mathcal{T}^\ast e_j}_{j=1}^n$ becomes a Riemann-orthonormal tangent frame with respect to $\mathcal{T}^\ast \met_t$ for each leaf $\Sigma_{\mathcal{T}(t)}$. Using all these ingredients shows 
\begin{equation}\label{midestep2wellpos1}
\begin{split}
\mathcal{T}_\ast \mathsf{v} &= - \frac{1}{N\circ\mathcal{T}}\mathcal{T}_\ast \partial_{t}=:\widetilde{\mathsf{v}}\\
\mathcal{T}^{\ast}\circ\left(\upbeta \nabla_{\mathsf{v}}\right)\circ\mathcal{T}^\ast v &= \cliff{\mathcal{T}_{\ast}\mathsf{v}} \mathcal{T}^\ast \nabla_{\mathsf{v}} \left(\mathcal{T}^{\ast}v\right)=\cliff{\widetilde{\mathsf{v}}}\nabla_{\widetilde{\mathsf{v}}}\mathcal{T}^{\ast}\mathcal{T}^{\ast}v = \widetilde{\upbeta}\nabla_{\widetilde{\mathsf{v}}}v \\
\mathcal{T}^\ast \circ\left(\upbeta A_t \right)\circ\mathcal{T}^\ast v &=  \widetilde{\upbeta}\sum_{j=1}^n \cliff{\mathcal{T}_\ast e_j} \nabla_{\mathcal{T}_\ast e_j} v = \widetilde{\upbeta} \widetilde{A}_t v \\
\mathcal{T}^\ast \circ \left(\upbeta H_t\right)\circ \mathcal{T}^\ast v &= \widetilde{\upbeta}\sum_{j=1}^n\mathcal{T}^\ast \met_t\left(\widetilde{\wein}(e_j),e_j\right) v =  \widetilde{\upbeta} \widetilde{H}_t v \\
\Rightarrow\,\,(\mathcal{T}^\ast \circ D_{\pm} \circ \mathcal{T}^\ast) v &= -\widetilde{\upbeta}\left(\nabla_{\widetilde{\mathsf{v}}}+\Imag \widetilde{A}_t v-\frac{n}{2} \widetilde{H}_t\right) v \quad
\end{split}
\end{equation}
where the tilded quantities are Clifford multiplication, the Weingarten map and the mean curvature with respect to the now future poining vector $\widetilde{\mathsf{v}}$, orthonormal to each hypersurface.
\end{proof}
The third line in \clef{midestep2wellpos1} implies that $A^{E}_t$ is invariant under this time-reversing, so $(A^{E}_t)^2$ and $(\mathcal{A}^{E}_t)^2$ do as well. The Lichnerowicz formula then shows that the spinorial Laplacian satisfies
\begin{equation*}
\mathcal{T}^\ast \circ (\Nabla{\spinb_{E}(\Sigma_t)}{})^\ast\Nabla{\spinb_{E}(\Sigma_t)}{}\circ \mathcal{T}^\ast=(\Nabla{\spinb_{E}(\Sigma_{\mathcal{T}(t)})}{})^\ast\Nabla{\spinb_{E}(\Sigma_{\mathcal{T}(t)})}{}
\end{equation*} 
and, as in the proof of \Cref{lemenest1} (2), one gains
\begin{equation}\label{timelambda}
\Lambda^s_{\mathcal{T}(t)}=\mathcal{T}^\ast\Lambda^s_{t}\mathcal{T}^\ast \quad\forall\,s\in \mathbb{R}\,\, ,t\in \timef(M)\quad.
\end{equation}

\subsection{Energy estimates}\label{chap:cauchy-sec:wellpos-2}

Suppose $M$ is spatially compact, i.e. the Cauchy hypersurface $\Sigma$ is compact. The \textit{s-energy along} $\Sigma$ of a sufficiently differentiable section $u$ of a vector bundle $E\rightarrow M$ is
\begin{equation*}
\mathcal{E}_s(u,\Sigma):= \norm{u\vert_\Sigma}{H^s(E\vert_\Sigma)}^2 \quad .
\end{equation*}
If we consider a non-compact, but complete Cauchy hypersurface $\Sigma$, we have to consider sections with compact-like support properties. Spatially compactness of a section $u$ with support in $\Jlight{}(K)$, $K\subset\Sigma$ compact, implies compactness of $\supp{u}\cap \Sigma$ and thus for any slice $\Sigma_t$. We apply the doubling procedure with respect to this compact subset and gain a compact Cauchy hypersurface $\widetilde{\Sigma}$ and an extended vector bundle $\widetilde{E}$ over $\widetilde{\Sigma}$ which restricts to $E\vert_\Sigma$ over the compact subset. The \textit{s-energy} along $\Sigma$ of a sufficiently differentiable section $u$ then becomes
\begin{equation}\label{senergy}
\mathcal{E}_s(u,\Sigma):= \norm{\tilde{u}\vert_{\widetilde{\Sigma}}}{H^s(\widetilde{\Sigma},\widetilde{E})}^2
\end{equation}
for any $s \in \R$. The following statement is the pendant of \cite[Thm.8]{BaerWafo} for the Dirac equation with spinor sections of positive chirality. The proof contains a similar argumentation, but since only one initial value is given and no constraint on the mean curvature is proposed, we had to show \Cref{lemenest1} as well as \Cref{lemenest1ext} in advance. 
\begin{prop}\label{enesttheorem}
Let $I \subset \timef(M)$ be a closed interval, $K\subset M$ compact and $s\in \R$; there exists a constant $C>0$, depending on $K$ and $s$ such that
\begin{equation}\label{enesttheoremform}
\mathcal{E}_s(u,\Sigma_{t_1}) \leq \mathcal{E}_s(u,\Sigma_{t_0})\expe{C(t_1-t_0)}+\int_{t_0}^{t_1}\expe{C(t_1-\tau)}\norm{D^{E}_{\pm}u\vert_{\Sigma_\tau}}{H^s_\loc(\spinb_{E}(\Sigma_\tau))}^2 \differ \tau
\end{equation}
applies for all $t_0,t_1 \in I$ with $t_0 < t_1$ and for all $u \in FE^{s+1}_{\scomp}(M,\timef,\spinb_{E}^{\pm}(M))$ with support $\supp{u}\subset \Jlight{}(K)$ and $D^{E}_{\pm}u \in FE^{s}_{\scomp}(M,\timef,\spinb_{E}^{\mp}(M))$. 
\end{prop}
\begin{proof}
It is enough to prove the statement in detail for the untwisted case and for spinor solutions of positive chirality. We point out the necessary changes for the twisted case and for solutions of negative chirality.\\
\\
W.l.o.g. we assume $M$ to be spatially compact, i.e. every leaf $\Sigma_t$ is closed; if otherwise, one applies the doubling procedure for each hypersurface. The Dirac operator is decomposed into a tangential and normal part with respect to each hypersurface $\Sigma_t$: $D_{+}=-\upbeta\left(\nabla_\mathsf{v}+B_t\right)$ with $B_t:=B_{t,+}$ as in \Cref{lemenest1} (3) (\Cref{lemenest1ext} (2) for the twisted case). Rewriting this as the covariant derivative $\nabla_{\partial_t}=-N_t \nabla_{\mathsf{v}}$ leads to $\left.\nabla_{\partial_t} u\right\vert_{\Sigma_t} = \left. N_t\upbeta D_{+}u \right\vert_{\Sigma_t} + \left. N_t B_t u \right\vert_{\Sigma_t}$. $B_t$ is a differential operator of order at most 1, acting in tangential direction; the preassumption $u \in FE^{s+1}_{\scomp}(M,\timef,\spinb^{+}(M))$ implies $u\vert_{\Sigma_t} \in H^{s+1}_{\loc}(\spinb^{+}(\Sigma_t))$ and thus $N_t B_t u\vert_{\Sigma_t} \in H^{s}_\loc(\spinb^{+}(\Sigma_t))$, implying $N_t B_t u \in FE^s_{\scomp}(M,\timef,\spinb^{+}(M))$. Since $D_{+}u \in FE^{s}_{\scomp}(M,\timef,\spinb^{-}(M))$ by preassumption, the first part of the right-hand side satisfies $\upbeta N_t D_{+}u \in FE^{s}_{\scomp}(M,\timef,\spinb^{+}(M))$. Thus, the covariant derivative with respect to $\mathsf{v}$ along any hypersurface $\Sigma_t$ is a Sobolev section in $H^s_\loc(\spinb^{+}(\Sigma_t))$ and therefore $\nabla_{\partial_t} u \in C^0(\timef(M),H^s_\loc(\spinb^{+}(\Sigma_\bullet))$, implying $u \in C^1(\timef(M),H^s_\loc(\spinb^{+}(\Sigma_\bullet))$. This time-differentiability and the continuity of the norm shows that the map $t \mapsto \mathcal{E}_s(u,\Sigma_t)$ is differentiable. Differentiation of $\mathcal{E}_s(u,\Sigma_t)$ with respect to $t$ leads with \clef{variationtIntegral} to
\begin{equation*}
\frac{\differ}{\differ t}\mathcal{E}_s(u,\Sigma_{t}) = n\dscal{1}{L^2(\spinb(\Sigma_t))}{H_t \Lambda^s_t u}{\Lambda^s_t u}-\int_{\Sigma_t}\mathsf{v} \dscal{1}{\spinb(\Sigma_t)}{\Lambda_t^s u}{\Lambda_t^s u} \dvol{\Sigma_t}
\end{equation*}
where $u$ is evaluated on the hypersurface $\Sigma_t$ and $\phi=1$ since every hypersurface is an (artificially) closed submanifold. We choose the connection to be compatible with the bundle metric and obtain
\begin{align*}
\frac{\differ}{\differ t}\mathcal{E}_s(u,\Sigma_{t}) &=  n\dscal{1}{L^2(\spinb(\Sigma_t))}{H_t \Lambda^s_t u}{\Lambda^s_t u}-2 \Re\mathfrak{e}\left\lbrace\dscal{1}{L^2(\spinb(\Sigma_t))}{\Lambda_t^s u}{\nabla_{\mathsf{v}}\Lambda_t^s u} \right\rbrace \\
&= n\dscal{1}{L^2(\spinb(\Sigma_t))}{H_t \Lambda^s_t u}{\Lambda^s_t u}-2 \Re\mathfrak{e}\left\lbrace\dscal{1}{L^2(\spinb(\Sigma_t))}{\Lambda_t^s u}{[\nabla_{\mathsf{v}},\Lambda_t^s] u} \right\rbrace\\
&\quad\quad-2 \Re\mathfrak{e}\left\lbrace\dscal{1}{H^s(\spinb(\Sigma_t))}{u}{\nabla_{\mathsf{v}} u} \right\rbrace \\
&= n\dscal{1}{L^2(\spinb(\Sigma_t))}{H_t \Lambda^s_t u}{\Lambda^s_t u}+\norm{u}{H^s(\spinb(\Sigma_t))}^2+\norm{[\nabla_{\mathsf{v}},\Lambda_t^s] u}{L^2(\spinb(\Sigma_t))}^2\\
&\quad\quad-\norm{(\Lambda^s_t+[\nabla_{\mathsf{v}},\Lambda_t^s]) u}{L^2(\spinb(\Sigma_t))}^2-2 \Re\mathfrak{e}\left\lbrace\dscal{1}{H^s(\spinb(\Sigma_t))}{u}{\nabla_{\mathsf{v}} u} \right\rbrace \\
\end{align*}
where one has used polarisation identities of the real parts. We apply a Sobolev estimate for $\nabla_{\mathsf{v}}$ as first order operator along $\Sigma_t$ and the compactness of the hypersurface justifies the use of \Cref{lemenest1} (2) (or \Cref{lemenest1ext} (1) for the twisted case). With \clef{extendB} (or its twisted analogue in \Cref{lemenest1ext} (2)) and another polarisation identity, the calculation goes on as follows:
\begin{align*}
\frac{\differ}{\differ t}\mathcal{E}_s(u,\Sigma_{t}) &\leq n\dscal{1}{L^2(\spinb(\Sigma_t))}{H_t \Lambda^s_t u}{\Lambda^s_t u}+c_2\norm{u}{H^s(\spinb(\Sigma_t))}^2-2 \Re\mathfrak{e}\left\lbrace\dscal{1}{H^s(\spinb(\Sigma_t))}{u}{\nabla_{\mathsf{v}} u} \right\rbrace \\
&= c_2 \norm{u}{H^s(\spinb(\Sigma_t))}^2+n\dscal{1}{L^2(\spinb(\Sigma_t))}{H_t \Lambda^s_t u}{\Lambda^s_t u} \\
&\quad\quad+2 \Re\mathfrak{e}\left\lbrace\dscal{1}{H^s(\spinb(\Sigma_t))}{u}{\upbeta D_{+} u}+\dscal{1}{H^s(\spinb(\Sigma_t))}{u}{B_t u} \right\rbrace \\
&= c_2 \norm{u}{H^s(\spinb(\Sigma_t))}^2+n\dscal{1}{L^2(\spinb(\Sigma_t))}{H_t \Lambda^s_t u}{\Lambda^s_t u} \\
&\quad\quad+2 \Re\mathfrak{e}\left\lbrace \dscal{1}{L^2(\spinb(\Sigma_t))}{\Lambda^s_t u}{B_t \Lambda^s_t u} + \dscal{1}{L^2(\spinb(\Sigma_t))}{\Lambda^s_t u}{[\Lambda^s_t,B_t] u} \right\rbrace \\
&\quad\quad+2 \Re\mathfrak{e}\left\lbrace\dscal{1}{L^2(\spinb(\Sigma_t))}{\Lambda^s_t u}{\upbeta \Lambda^s_t D_{+} u}\right\rbrace \\
&\leq  (c_2+2) \norm{u}{H^s(\spinb(\Sigma_t))}^2+\norm{\upbeta \Lambda^s_t D_{+} u}{L^2(\spinb(\Sigma_t))}^2+\norm{[\Lambda^s_t,B_t] u}{L^2(\spinb(\Sigma_t))}^2\\
&\leq c \norm{u}{H^s(\spinb(\Sigma_t))}^2+\norm{D_{+} u}{H^s(\spinb(\Sigma_t))}^2= c \mathcal{E}_s(u,\Sigma_{t}) +\norm{D_{+}u}{H^s(\spinb(\Sigma_t))}^2 \quad.
\end{align*}
Formula \clef{isombeta} has been used in the last step which has to be replaced with \clef{isombetatwist} for the twisted case. The remaining commutator acts as properly supported pseudo-differential operator of order $(s+1)$ such that the last inequality follows with the continuous embedding of Sobolev spaces. The compact support assumption on $u\vert_{\Sigma_t}$ for any $t$ can be weakend to $u\vert_{\Sigma_t}$ being a local Sobolev section. The above steps can be repeated for the compactly supported function $\phi u\vert_{\Sigma_t}$ for any smooth and compactly supported function $\phi$ on $\Sigma_t$ such that one has to consider the Sobolev norms for local Sobolev sections in \clef{senergy}:
\begin{equation*}
\frac{\differ}{\differ t}\mathcal{E}_s(u,\Sigma_{t}) \leq c \mathcal{E}_s(u,\Sigma_{t}) +\norm{D_{+}u}{H^s_\loc(\spinb(\Sigma_t))}^2 \quad.
\end{equation*}
Before applying Grönwall's inequality, we have to take a closer look on the constant $c$ beforehand. It is known from the local theory of pseudo-differential operators, acting on Sobolev sections, that the constant depends on the Sobolev degree $s$, the compact superset containing the support of the section, and on the derivatives of the metric $\met_t$ as Jacobi's formula is applied for derivatives on the volume form prefactor in \clef{int0}. The hypersurface metrices and their derivatives depend smoothly on the time parameter $t$ such that the constant depends also smoothly on $t$. The computed estimate of $\frac{\differ}{\differ t}\mathcal{E}_s(u,\Sigma_{t})$ then finally takes the form
\begin{equation*}  
\frac{\differ}{\differ t}\mathcal{E}_s(u,\Sigma_{t}) \leq c(\Vert\met_t\Vert_{K,s})\mathcal{E}_s(u,\Sigma_{t}) +\norm{D_{+}u}{H^s_\loc(\spinb(\Sigma_t))}^2
\end{equation*}
with $\Vert\cdot\Vert_{K,s}$ as seminorm on $C^\infty((T^\ast \Sigma_t)^{\otimes 2})$. Grönwall's Lemma gives
\begin{equation*}
\mathcal{E}_s(u,\Sigma_{t_1}) \leq \mathcal{E}_s(u,\Sigma_{t_0})\expe{\int_{t_0}^{t_1} c(\Vert\met_t\Vert_{K,s}) \differ t}+\int_{t_0}^{t_1}\expe{\int_{\tau}^{t_1}c(\Vert\met_t\Vert_{K,s})\differ t}\norm{D_{+}u\vert_{\Sigma_\tau}}{H^s_\loc(\spinb(\Sigma_\tau))}^2 \differ \tau \, .
\end{equation*}
The extreme value theorem on closed time intervals, applied on $c$, leads to the stated result for $D_{+}$ where $C=C(\Vert \met_\bullet \Vert_{\Jlight{}(K),m(s)})$ is the maximum on $[t_0,t_1]$.\\ 
\\
For $u$ having negative chirality we observe that the only influence of the chirality appears in expressing the spinorial covariant derivative with respect to $\mathsf{v}$ in terms of $D_{-}$ and operators on the hypersurfaces where $B_t$ has to be replaced by $B_{t,-}$. But since the result of \Cref{lemenest1} (2), (3) (\Cref{lemenest1ext} (1) and (2) for the twisted case) are independent of the chirality, the proof carries over for this case as well.
\end{proof}
The following two corollaries are the pendents of \cite[Cor.10/11]{BaerWafo} for the Dirac equation.
\begin{cor}\label{corenest1}
Given $\intervallc{t_0}{t_1}{} \subset \timef(M)$, $\tau \in \timef(M)$, $K \subset M$ compact and $s \in \R$; there exists a $C>0$, depending on $K$ and $s$ such that
\begin{equation*}
\mathcal{E}_s(u,\Sigma_{t})\leq C \left(\mathcal{E}_s(u,\Sigma_{\tau})+ \norm{D^{E}_{\pm} u}{\intervallc{t_0}{t_1}{},\Jlight{}(K),s}^2\right)
\end{equation*}
is valid for all $t \in \intervallc{t_0}{t_1}{}$, for all $u \in FE^{s+1}_{\scomp}(M,\timef,D^{E}_{\pm})$ with $\supp{u}\subset \Jlight{}(K)$ and $D^{E}_{\pm}u \in FE^s_{\scomp}(M,\timef,\spinb_{E}^{\mp}(M))$.
\end{cor}
\begin{proof}
As in the proof of \Cref{enesttheorem}, we focus on the untwisted Dirac operator to keep notation simple.\\
\\
We assume again that each leaf is closed, otherwise one extends again everything to a suitable double. We choose $\tau \in \intervallc{t_0}{t_1}{}$, otherwise we take $t_0,t_1 \in \timef(M)$ in such a way that it is true. Suppose first $t \in \intervallc{\tau}{t_1}{}$; since $u\in FE^{s+1}_{\scomp}(M,\timef,D_{\pm})\subset FE^{s+1}_{\scomp}(M,\timef,\spinb^{\pm}(M))$, all preassumptions from this corollary coincide with the one from \Cref{enesttheorem} such that \clef{enesttheoremform} holds. By assumption, $D_{\pm}u \in L^2_{\loc,J(K)}(\timef,H^s(\spinb^{\mp}(\Sigma_{\bullet})))$ implies integrability of the map $\lambda \mapsto \norm{D_{\pm}u\vert_{\Sigma_{\lambda}}}{H^s(\spinb(\Sigma_{\lambda}))}$ such that
\begin{eqnarray*}
\mathcal{E}_s(u,\Sigma_{t})&\leq& c\mathcal{E}_s(u,\Sigma_\tau)\expe{c(t-\tau)}+\int_{\tau}^{t}\expe{c(t-\lambda)}\norm{D_{\pm}u\vert_{\Sigma_\lambda}}{H^s(\spinb(\Sigma_{\lambda}))}^2 \differ \lambda \\
&\leq& C \left(\mathcal{E}_s(u,\Sigma_\tau)+\norm{D_{\pm} u}{\intervallc{t_0}{t_1}{},\Jlight{}(K),s}^2\right) \quad.
\end{eqnarray*}
We use a time-reversal argument with $\mathcal{T}$ from \clef{timerevmap}, applied on $\intervallc{t_0}{\tau}{}$. \Cref{timeinvarinace} ensures that $\mathcal{T}^\ast u$ solves the backward time Dirac equation with time-reversed data such that the proof of \Cref{enesttheorem} can be repeated and applied for this situation:
\begin{equation*}
\mathcal{E}_s(\mathcal{T}^{\ast}u,\Sigma_{\mathcal{T}(t)})\leq  C \left(\mathcal{E}_s(\mathcal{T}^{\ast}u,\Sigma_{\mathcal{T}(\tau)})+\int_{\mathcal{T}(\tau)}^{\mathcal{T}(t)}\norm{(\widetilde{D}_{\pm}\mathcal{T}^\ast u)\vert_{\Sigma_\lambda}}{H^s(\spinb(\Sigma_{\lambda}))}^2 \differ \lambda\right) \quad
\end{equation*}
where $\widetilde{D}_{\pm}=\mathcal{T}^\ast \circ D_{\pm} \circ \mathcal{T}^\ast$. Because $\mathcal{T}$ is an involution, the $s$-energy and the last integral over the inhomogeneity are invariant under this time-orientation reversion according to \clef{timelambda}, leading to $\mathcal{E}_s(\mathcal{T}^{\ast}u,\Sigma_{\mathcal{T}(t)})=\mathcal{E}_s(u,\Sigma_t)$ with the help of the transformation law of integration. A similar integration shows the latter invariance by substitution:
\begin{equation*}
\int_{\mathcal{T}(\tau)}^{\mathcal{T}(t)}\norm{(\widetilde{D}_{\pm}\mathcal{T}^\ast u)\vert_{\Sigma_\lambda}}{H^s(\spinb(\Sigma_{\lambda}))}^2 \differ \lambda = \int_{\tau}^{t}\norm{(D_{\pm} u)\vert_{\Sigma_\rho}}{H^s(\spinb(\Sigma_{\rho}))}^2 \differ \rho\quad.\qedhere
\end{equation*}
%
\end{proof}
The following conclusion can be interpreted as \textit{uniqueness of the Cauchy problem for the Dirac equation}.
\begin{cor}\label{corenest2}
$u \in FE^{s}_\scomp(M,\timef,D^{E}_{\pm})$ is uniquely determined by the inhomogeneity $D^{E}_{\pm}u$ and the initial condition $u\vert_{\Sigma_t}$ on a hypersurface $\Sigma_t$ for any $t \in \timef(M)$.
\end{cor}
\begin{proof}
The argumentation carries over literally from the proof of \cite[Cor.11]{BaerWafo} where Corollary 10 in the reference is replaced by \Cref{corenest1}.
\end{proof}

\subsection{Well-posedness of $D^{E}_{\pm}$}\label{chap:cauchy-sec:wellpos-3}

After all preparations, the well-posedness of the inhomogeneous Cauchy problem for the Dirac equation for spinor fields with positive chirality can be proven in the same way as it has been done for the wave equation in \cite{BaerWafo}. The proof from this reference carries over literally to the setting of our interest where the used corollaries 10 and 11 are replaced by \Cref{corenest1} as well as by the uniqueness of solutions of the Dirac equation in \Cref{corenest2}. But we repeat the argument for the sake of completeness and later argumentation. We consider the map which generates the initial value problem onto any hypersurface $\Sigma_t$, associated to $D^{E}_{\pm}$:
\begin{equation}\label{inivpmapgen}
\begin{array}{lccc}
\rest{t}\oplus D^{E}_{\pm}:& C^\infty(\spinb^{\pm}_{E}(M))&\rightarrow & C^\infty(\spinb^{\pm}_{E}(\Sigma_t))\oplus C^\infty(\spinb^{\mp}_{E}(M))\\
 & u &\mapsto & (u\vert_{\Sigma_t},D^{E}_{\pm}u)\quad.
\end{array}
\end{equation}
\begin{theo}\label{inivpwell}
For a fixed $t \in \timef(M)$ and $s \in \R$ the map \clef{inivpmapgen} extends to 
\begin{equation}\label{inivpmap}
\mathsf{res}_t \oplus D^{E}_{\pm} \,\,:\,\, FE^s_\scomp(M,\timef,D^{E}_{\pm}) \,\,\rightarrow\,\, H^s_\comp(\spinb^{\pm}_{E}(\Sigma_t))\oplus L^2_{\loc,\scomp}(\timef(M),H^s_\loc(\spinb^{\mp}_{E}(\Sigma_\bullet)))\, 
\end{equation}
which is an isomorphism of topological vector spaces.
\end{theo}
This result has been proven in \cite[Thm 4.7/4.13]{OD} for the untwisted Dirac operator.
\begin{proof}
One first checks the continuity of the map $\rest{t}\oplus D^{E}_{\pm}$, induced by the continuity of both summands: $FE^s_\scomp(M,\timef,\spinb^{\pm}_{E}(M))$ is by definition the union of all continuous functions from $\timef(M)$ to $H^s_\loc(\spinb^{\pm}_{E}(\Sigma_\bullet))$ with spatially compact support in $\mathcal{K} \subset M$. An intersection of $\mathcal{K}$ with any Cauchy hypersurface in the foliation of $M$ is a compact subset and since $H^s_\comp(\spinb^{\pm}_{E}(\Sigma_{\bullet}))$ is also defined as union over all compact subsets in any slice, it is enough to consider the restriction onto a fixed hypersurface $\Sigma_t$ as map between $C^0_\mathcal{K}(\timef(M),H^s_\loc(\spinb^{\pm}_{E}(\Sigma_\bullet)))$ and $H^s_{\mathcal{K}\cap\Sigma_t}(\spinb^{\pm}_{E}(\Sigma_t))$. The continuity follows immediatly from the estimate
\begin{equation*}
\norm{\rest{t}u}{H^s(\mathcal{K}\cap \Sigma_t ,\spinb_{E}(\Sigma_t))} \leq \max_{\tau \in \timef(M)}\SET{\norm{u\vert_{\Sigma_\tau}}{H^s(\mathcal{K}\cap \Sigma_{\tau},\spinb_{E}(\Sigma_t))}}=\norm{u}{\timef(M),\mathcal{K},0,s}
\end{equation*}
with the norm \clef{snclk} on $C^0_{\mathcal{K}}(\timef(M),H^s_\loc(\spinb^{\pm}_{E}(\Sigma_\bullet)))$. The two inclusion mappings $C^0_{\mathcal{K}}(\timef(M),H^s_\loc(\spinb^{\pm}_{E}(\Sigma_\bullet))) \hookrightarrow C^0_{\scomp}(\timef(M),H^s_\loc(\spinb^{\pm}_{E}(\Sigma_\bullet)))$ and $H^s_{\mathcal{K}\cap\Sigma_\bullet}(\spinb^{\pm}_{E}(\Sigma_\bullet)) \hookrightarrow H^s_\comp(\spinb^{\pm}_{E}(\Sigma_\bullet))$ are continuous and the restriction map between $FE^s_{\scomp}(M,\timef,D^{E}_{\pm})$ and $H^s_\comp(\spinb^{\pm}_{E}(\Sigma_t))$ for a fixed $\Sigma_t$ becomes continuous. $D^{E}_{\pm}$ as maps from $FE^s_\scomp(M,\timef,D^{E}_{\pm})$ to $L^2_{\loc,\scomp}(\timef(M),H^s_\loc(\spinb^{\mp}_{E}(\Sigma_\bullet)))$ are continuous on $FE^s_\scomp(M,\timef,D^{E}_{\pm})$, implying \clef{inivpmapgen} to be continuous on the considered domains and ranges.\\
\\
It is left to show that \clef{inivpmap} are bijective with continuous inverse. We take a $K \subset \Sigma_t$ compact for any but fixed $t \in \timef(M)$. The well-posedness of the Cauchy problem for the Dirac equation with smooth and compactly supported initial data on $\Sigma_t$ (see \cite[Thm.4]{AndBaer}) states that for given $u_0 \in C^\infty_K(\spinb^{\pm}_{E}(\Sigma_t))$ and $f \in C^\infty_{\Jlight{}(K)}(\spinb^{\mp}_{E}(M))$ there exist solutions $u\in C^\infty(\spinb^{\pm}_{E}(M))$ of the Dirac equation with inhomogeneity $f$ and initial value $u_0=u\vert_{\Sigma_t}$ which have support in $\Jlight{}(K)$ by finite propagation speed. \Cref{corenest1} then implies for any subinterval $I \subset \timef(M)$ and $t \in I$ a fixed initial time
\begin{eqnarray*}
\norm{u}{I,\Jlight{}(K),0,s}^2 &=& \max_{\tau \in I}\Bigl\{\mathcal{E}(u,\Sigma_\tau)\Bigr\} \leq C \max_{\tau \in I}\SET{\norm{u\vert_{\Sigma_t}}{H^s_\loc(\spinb_{E}(\Sigma_t))}^2+\norm{D^{E}_{\pm}u}{I,\Jlight{}(K),s}^2}\\
&=&C \left(\norm{u_0}{H^s(\spinb_{E}(\Sigma_t))}^2+\norm{f}{I,\Jlight{}(K),s}^2\right) \quad. 
\end{eqnarray*} 
The estimation constant comes from the used corollary and thus it is not depending on the smooth data of the Cauchy problem. This result implies that the continuous map $(u_0,f)\mapsto u$ from the well-posedness of the Cauchy problem with smooth data can be extended to continuous maps 
\begin{equation*}
H^s_K(\spinb^{\pm}_{E}(\Sigma_t))\oplus L^2_{\loc,\Jlight{}(K)}(\timef(M),H^s_\loc(\spinb^{\mp}_{E}(\Sigma_\bullet)))\,\,\rightarrow\,\,C^0_{J(K)}(\timef(M),H^s_\loc(\spinb^{\pm}_{E}(\Sigma_\bullet)))\quad.
\end{equation*}
$\Jlight{}(K)$ is closed for $K \subset \Sigma_t \subset M$ compact and thus itself spatially compact such that
\begin{equation*}
H^s_\comp(\spinb^{\pm}_{E}(\Sigma_t))\oplus L^2_{\loc,\scomp}(\timef(M),H^s_\loc(\spinb^{\mp}_{E}(\Sigma_\bullet)))\,\,\rightarrow\,\,FE^s_\scomp(M,\timef,\spinb^{\pm}_{E}(M))
\end{equation*}
are continuous as the inclusions $H^s_K \hookrightarrow H^s_\comp$, $L^2_{\loc,\Jlight{}(K)}\hookrightarrow L^2_{\loc,\scomp}$ and $C^0_{\Jlight{}(K)}\hookrightarrow C^0_{\scomp}$ are continuous. The formal inverses are the maps
\begin{equation}\label{solmap}
H^s_\comp(\spinb^{\pm}_{E}(\Sigma_t))\oplus L^2_{\loc,\scomp}(\timef(M),H^s_\loc(\spinb^{\mp}_{E}(\Sigma_\bullet)))\,\,\rightarrow\,\,FE^s_\scomp(M,\timef,D^{E}_{\pm}) \quad.
\end{equation}
The composition \clef{inivpmap} after \clef{solmap} clearly gives the identity after fixing one hypersurface for the initial data. The converse composition starts with a solution from which one extracts the initial data and the inhomogeneity and solves again by \clef{solmap}. \Cref{corenest2} indicates uniqueness of the solution and the result coincides with the input, leading to the desired bijectivity of \clef{inivpmap}. The continuity of the inverses follows from the continuity of $D^{E}_{\pm}$ on $FE^s_{\scomp}(M,\timef,D^{E}_{\pm})$ and \clef{solmap}. Summarising all results shows that \clef{inivpmap} is indeed an isomorphism.  
\end{proof}
The well-posedness of the homogeneous Cauchy problem for the Dirac equation follows immediately.
\begin{cor}\label{homivpwell}
For a fixed $t \in \timef(M)$ and $s \in \R$ the map
\begin{equation*}
\mathsf{res}_t  \,\,:\,\, FE^s_\scomp\left(M,\timef,\kernel{D^{E}_{\pm}}\right) \,\,\rightarrow\,\, H^s_\comp(\spinb^{\pm}_{E}(\Sigma_t))
\end{equation*}
is an isomorphism of topological vector spaces.
\end{cor}
These two results lead to the following consequences.
\begin{cor}\label{corivp1}
For any $s\in \R$
\begin{itemize}
\item[(1)] $C^\infty_{\scomp}(\spinb^{\pm}_{E}(M)) \subset FE^s_\scomp(M,\timef,D^{E}_{\pm})$ is dense, 
\item[(2)] $C^\infty_{\scomp}(\spinb^{\pm}_{E}(M))\cap\kernel{D} \subset FE^s_\scomp\left(M,\timef,\kernel{D^{E}_{\pm}}\right)$ is dense.
\end{itemize}
\end{cor}
\begin{proof}
The proof can be taken from \cite[Cor.15]{BaerWafo} with the only difference that Theorem 13 and Corollary 14 in the reference are replaced by the well-posedness for the homogeneous and inhomogeneous Cauchy problem. 
\end{proof}
It follows that solutions of the Dirac equation have \textit{finite propagation speed}.
\begin{cor}
A solution $u$ with $u\vert_{\Sigma_t}=u_0$ for any leaf $\Sigma_t$, $t \in \timef(M)$, and $D^{E}_{\pm}u=f$ satisfies $\supp{u}\subset \Jlight{}(K)$ for a compact subset $K \subset M$, satisfying the support relation $\supp{u_0}\cup\supp{f} \subset K$.
\end{cor}
Another consequence of \Cref{corivp1} is an optimisation of the assumed regularity in \Cref{enesttheorem}.
\begin{cor}\label{corivp3}
The energy estimate \clef{enesttheoremform} already holds for all $t_0,t_1 \in I \subset \timef(M)$ with $t_0 < t_1$ and for all $u \in FE^s_\scomp(M,\timef,D^{E}_{\pm})$ with support $\supp{u}\subset \Jlight{}(K)$. 
\end{cor}
Another conclusion from \Cref{corivp1} is the independence of the Cauchy temporal function $\timef$ for those finite energy spinors which are solutions of the homogeneous Dirac equation.
\begin{cor}
Given two Cauchy temporal functions $\timef$ and $\timef'$ on $M$, then for all $s \in \R$ $$FE^s_\scomp\left(M,\timef,\kernel{D^{E}_{\pm}}\right)=FE^s_\scomp\left(M,\timef',\kernel{D^{E}_{\pm}}\right)\quad.$$
\end{cor}
\begin{proof}
The detailed, but involved proof can be taken from \cite[Cor.18]{BaerWafo} with the shown uniqueness of the Cauchy problem for the Dirac equation in \Cref{corenest2}.
\end{proof}
Thus, we can simplify notation to $FE^s_\scomp\left(M,\kernel{D^{E}_{\pm}}\right)$ to stress this independence of the temporal function.

\section{Wave evolution operators}\label{chap:Feynman}

The well-posedness of the homogeneous Cauchy problem in \Cref{homivpwell} motivates to define several \textit{(Dirac-)wave evolution operators}. This section is dedicated to study several properties of these wave evolution operators and in particular their property of being a Fourier integral operator.

\subsection{General properties}\label{chap:cauchy-sec:diracwave-1}

For latter reasons we will first consider the wave-evolution operators for the untwisted Dirac operators.
\begin{defi}\label{defevop}
For a globally hyperbolic manifold $M$ and $t_1,t_2 \in \timef(M)$ the (Dirac-) wave evolution operators for positive and negative chirality are the following isomorphisms of topological vector spaces
\begin{eqnarray*}
Q(t_2,t_1):=\rest{t_2}\circ(\rest{t_1})^{-1} &:& H^s_\comp(\spinb^{+}(\Sigma_{1}))\,\,\rightarrow\,\,H^s_\comp(\spinb^{+}(\Sigma_{2}))\,\, , \\
\tilde{Q}(t_2,t_1):=\rest{t_2}\circ(\rest{t_1})^{-1}&:&H^s_\comp(\spinb^{-}(\Sigma_{1}))\,\,\rightarrow\,\,H^s_\comp(\spinb^{-}(\Sigma_{2}))\,\, . 
\end{eqnarray*}
\end{defi}
The operator $Q(t_2,t_1)$ occurs in \cite{BaerStroh} for compact hypersurfaces and in \cite{BaerStroh2} for square-integrable sections on non-compact hypersurfaces. The wave evolution operators in our setting act between compactly supported Sobolev sections of any degree over non-compact, but complete hypersurfaces. The same properties of $Q$, as shown in \cite{BaerStroh}, are given for all introduced operators as well.
\newpage
\begin{lem}\label{propsofprop}
The following properties hold for any $s \in \R$ and $t,t_1,t_2,t_3 \in \timef(M)$:
\begin{itemize}
\item[(1)] $Q(t_3,t_2)\circ Q(t_2,t_1)=Q(t_3,t_1)$, 
\item[(2)] $Q(t,t)=\Iop{H^s_\comp(\spinb^{+}(\Sigma_t))}$ and $Q(t_1,t_2)=(Q(t_2,t_1))^{-1}$,
\item[(3)] $\tilde{Q}(t_3,t_2)\circ \tilde{Q}(t_2,t_1)=\tilde{Q}(t_3,t_1)$, 
\item[(4)] $\tilde{Q}(t,t)=\Iop{H^s_\comp(\spinb^{-}(\Sigma_t))}$ and $\tilde{Q}(t_1,t_2)=(\tilde{Q}(t_2,t_1))^{-1}$,
\item[(5)] For any $[t_1,t_2]\subset \timef(M)$ the operators $Q(t_2,t_1)$ and $\tilde{Q}(t_2,t_1)$ are unitary for $s=0$.
\end{itemize}
\end{lem}
\begin{proof} 
(1) to (4) follow by the same reasoning as in \cite{BaerStroh}. The well-posedness of the homogeneous Cauchy problem implies that any initial value $u_{\pm} \in H^s_\comp(\spinb^{\pm}(\Sigma_t))$ for $t \in \timef(M)$ is uniquely related to a finite energy spinor $\psi_{\pm} \in FE^s_{\scomp}(M,\kernel{D_{\pm}})$ for both chiralities such that $\psi_{\pm}\vert_{\Sigma_t}=u_{\pm}$. \Cref{boots} ensures that one even has $\psi_{\pm} \in C^1_{\scomp}(\spinb^{\pm}(M))$ for fixed $s > \frac{n}{2}+2$. Restricting to any compact time interval $[t_1,t_2]$ implies that $\psi_{\pm} \in C^1_{\comp}(\spinb^{\pm}(M))$ and \Cref{diracselfadprop} leads to the claim: for $\psi_{\pm}\in C^1_{\comp}(\spinb^{\pm}(M))\cap\kernel{D_{\pm}})$ we have
\begin{eqnarray*}
0&=& \int_{M\vert_{[t_2,t_1]}} \idscal{1}{\spinb(M)}{\Dirac \psi_{\pm}}{\psi_{\pm}}+ \idscal{1}{\spinb(M)}{\psi_{\pm}}{\Dirac \psi_{\pm}} \dvol{} \\
&\stackrel{\clef{diracselfad}}{=}& \int_{\Sigma_{2}} \dscal{1}{\spinb(\Sigma_2)}{\psi_{\pm}\vert_{\Sigma_{2}}}{\psi_{\pm}\vert_{\Sigma_{2}}} \dvol{\Sigma_{2}}-\int_{\Sigma_{1}} \dscal{1}{\spinb(\Sigma_{1})}{\psi_{\pm}\vert_{\Sigma_{1}}}{\psi_{\pm}\vert_{\Sigma_{1}}} \dvol{\Sigma_{1}} \quad.\\
\end{eqnarray*}
Using $\psi_{+}\vert_{\Sigma_2}=Q(t_2,t_1)\psi_{+}\vert_{\Sigma_1}=Q(t_2,t_1)u_{+}$ and $\psi_{-}\vert_{\Sigma_2}=\tilde{Q}(t_2,t_1)u_{-}$, we get 
\begin{eqnarray*}
0&=& \int_{\Sigma_{2}} \dscal{1}{\spinb(\Sigma_{2})}{Q(t_2,t_1)u_{+}}{Q(t_2,t_1)u_{+}} \dvol{\Sigma_{2}}-\int_{\Sigma_{t_1}} \dscal{1}{\spinb(\Sigma_{1})}{u_{+}}{u_{+}} \dvol{\Sigma_{1}} \\
&=&\norm{Q(t_2,t_1)u_{+}}{L^2(\spinb^{+}(\Sigma_{2}))}^2-\norm{u_{+}}{L^2(\spinb^{+}(\Sigma_{1}))}^2
\end{eqnarray*} 
and  
\begin{equation*}
\norm{\tilde{Q}(t_2,t_1)u_{-}}{L^2(\spinb^{-}(\Sigma_{2}))}^2=\norm{u_{-}}{L^2(\spinb^{-}(\Sigma_{1}))}^2 \quad.
\end{equation*}
We assumed w.l.o.g. that the compact supports of $\psi_{\pm}$ intersect both boundary hypersurfaces $\Sigma_{1}$ and $\Sigma_{2}$; if otherwise, one gets the trivial identifications $u_{\pm}=0$. 
Hence (5) is shown. 
\end{proof}
\begin{rem}\label{contTQ}
The map $t\mapsto \rest{t}$ is continuous for all $t\in \timef(M)$, implying the continuity of the map $t\mapsto Q(t,\cdot)$. \Cref{homivpwell} states that the solution depends continuously on the initial data and is continuous in $t$ due to the definition of finite energy spinors. Hence $t\mapsto \rest{t}^{-1}$ is continuous for all $t\in \timef(M)$ such that also $t\mapsto Q(\cdot,t)$ and finally the map $(t,s)\mapsto Q(t,s)$ are continuous for all $s,t \in \timef(M)$. The same reasoning transfers to the other introduced Dirac-wave evolution operator.     
\end{rem}
We can define the evolution operators for the twisted Dirac operators analogously:
\begin{eqnarray*}
Q^{E}(t_2,t_1):=\rest{t_2}\circ(\rest{t_1})^{-1} &:& H^s_\comp(\spinb^{+}_{E}(\Sigma_{1}))\,\,\rightarrow\,\,H^s_\comp(\spinb^{+}_{E}(\Sigma_{2}))\,\, , \\
\tilde{Q}^{E}(t_2,t_1):=\rest{t_2}\circ(\rest{t_1})^{-1}&:&H^s_\comp(\spinb^{-}_{E}(\Sigma_{1}))\,\,\rightarrow\,\,H^s_\comp(\spinb^{-}_{E}(\Sigma_{2}))\,\, . 
\end{eqnarray*}
\Cref{propsofprop} and \Cref{contTQ} still hold if we replace $Q$ and $\tilde{Q}$ with $Q^E$ or rather $\tilde{Q}^E$.

\subsection{Evolution operators as FIO}\label{chap:cauchy-sec:diracwave-2}

The proof of the following result is the main task of this subsection and parts of Lemma 2.6 in \cite{BaerStroh}.
%
%
\begin{theo}\label{Qfourier} 
The operators $Q$, $\tilde{Q}$, $Q^{E}$ and $\tilde{Q}^{E}$ satisfy
\begin{eqnarray*}
&&Q(t_2,t_1)\in \FIO{0}_{\mathsf{prop}}(\Sigma_{1},\Sigma_{2};\mathsf{C}'_{1\rightarrow 2};\Hom(\spinb^{+}(\Sigma_1),\spinb^{+}(\Sigma_2)))\\
&& \tilde{Q}(t_2,t_1)\in \FIO{0}_{\mathsf{prop}}(\Sigma_{1},\Sigma_{2};\mathsf{C}'_{1\rightarrow 2};\Hom(\spinb^{-}(\Sigma_1),\spinb^{-}(\Sigma_2)))\\
&& Q^{E}(t_2,t_1)\in \FIO{0}_{\mathsf{prop}}(\Sigma_{1},\Sigma_{2};\mathsf{C}'_{1\rightarrow 2};\Hom(\spinb^{+}_{E}(\Sigma_1),\spinb^{+}_{E}(\Sigma_2)))\\
&& \tilde{Q}^{E}(t_2,t_1)\in \FIO{0}_{\mathsf{prop}}(\Sigma_{1},\Sigma_{2};\mathsf{C}'_{1\rightarrow 2};\Hom(\spinb^{-}_{E}(\Sigma_1),\spinb^{-}_{E}(\Sigma_2)))
\end{eqnarray*}
for any fixed time interval $[t_1,t_2]\subset \timef(M)$ and with canonical graphs 
\begin{equation}\label{Qcanrel}
\begin{split}
\mathsf{C}_{1\rightarrow 2} &=\mathsf{C}_{1\rightarrow 2\vert +}\sqcup \mathsf{C}_{1\rightarrow 2\vert -} \quad\text{where}\\
\mathsf{C}_{1\rightarrow 2\vert \pm} &=\Bigl\{((x_{\pm},\xi_{\pm}),(y,\eta))\in \dot{T}^\ast \Sigma_2\times \dot{T}^\ast \Sigma_1\,\vert\, (x_{\pm},\xi_{\pm})\sim (y,\eta)\Bigr\}
\end{split}
\end{equation} 
with respect to the lightlike (co-)geodesic flow as canonical relation; their principal symbols are
\begin{align}
\mathrel{\phantom{XX}}{\sigma}_{0}(Q)(x,\xi_{\pm};y,\eta)&=\pm\frac{1}{2} \norm{\eta}{\met_{t_1}(y)}^{-1}\left(\mp\norm{\xi_{\pm}}{\met_{t_2}(x)}\upbeta+\Cliff{t_2}{\xi^{\sharp}_{\pm}}\right)\circ \mathpzc{P}^{\spinb(M)}_{(x,\varsigma_{\pm})\leftarrow (y,\zeta_{\pm})} \circ \upbeta\label{Qprincsymbol}\\
{\sigma}_{0}(Q^{E})(x,\xi_{\pm};y,\eta)&= {\sigma}_{0}(Q)(x,\xi_{\pm};y,\eta)\otimes\left[\Iop{E\vert_{\Sigma_2}} \circ   \mathpzc{P}^{E}_{(x,\varsigma_{\pm})\leftarrow (y,\zeta_{\pm})} \circ \Iop{E\vert_{\Sigma_1}} \right]\nonumber
\end{align}
and
\begin{align}
{\sigma}_{0}(\tilde{Q})(x,\xi_{\pm};y,\eta)&=\pm\frac{1}{2} \norm{\eta}{\met_{t_1}(y)}^{-1}\left(\mp\norm{\xi_{\pm}}{\met_{t_2}(x)}\upbeta-\Cliff{t_2}{\xi^{\sharp}_{\pm}}\right) \circ \mathpzc{P}^{\spinb(M)}_{(x,\varsigma_{\pm})\leftarrow (y,\zeta_{\pm})} \circ \upbeta\label{Qnegprincsymbol}\\
{\sigma}_{0}(\tilde{Q}^{E})(x,\xi_{\pm};y,\eta) &= {\sigma}_{0}(\tilde{Q})(x,\xi_{\pm};y,\eta)\otimes\left[\Iop{E\vert_{\Sigma_2}} \circ   \mathpzc{P}^{E}_{(x,\varsigma_{\pm})\leftarrow (y,\zeta_{\pm})} \circ \Iop{E\vert_{\Sigma_1}} \right] \nonumber
\end{align}
where $(y,\zeta_{\pm})\in T^\ast_{\Sigma_{1}}M$ and $(x,\varsigma_{\pm})\in T^\ast_{\Sigma_{2}}M$ restrict to $(y,\eta)$ and $(x,\xi_{\pm})$ respectively.
\end{theo}
$\mathpzc{P}^{\spinb(M)}_{(x,\varsigma_{\pm})\leftarrow (y,\zeta_{\pm})}$ and $\mathpzc{P}^{E}_{(x,\varsigma_{\pm})\leftarrow (y,\zeta_{\pm})}$ denote the parallel transport from $(y,\zeta_{\pm})$ to $(x,\varsigma_{\pm})$ with respect to the spinorial and respectively twisting bundle covariant derivative, $\norm{\bullet}{\met_{t}}$ a norm for covectors and $\sharp$ the sharp isomorphism, induced by the dual metric of $\met_t$ for fixed $t$, and $T^\ast_{\Sigma_{t}}M:=T^\ast_pM$ for any $p \in \Sigma_t$.
\begin{proof}
W.l.o.g. we assume that $M$ is temporal compact with $\timef(M)=[t_1,t_2]$; if otherwise, we have to replace $M$ with the temporal restriction $M\vert_{[t_1,t_2]}$. In order to describe $Q^{E}$ and $\tilde{Q}^{E}$ as FIOs, one uses the fact that $(\Dirac^{E})^2$ and thus $D^{E}_{\mp}D^{E}_{\pm}$ are normally hyperbolic. \Cref{cauchynormhyphom} assures that the homogeneous Cauchy problems
\begin{equation}\label{homcauchysol1}
\begin{split}
D^{E}_{+}D^{E}_{-}v&=0\,\, , \quad \rest{\Sigma_1}v=0\,\,,\quad \rest{\Sigma_1}(-\Nabla{\spinb_{E}(M)}{\mathsf{v}})v=g^{-}\\
D^{E}_{-}D^{E}_{+}w&=0\,\, , \quad \rest{\Sigma_1}w=0\,\,,\quad \rest{\Sigma_1}(-\Nabla{\spinb_{E}(M)}{\mathsf{v}})w=g^{+}
\end{split}
\end{equation}
have unique solutions $v$ and $w$ as spinor fields with negative and respectively positive chirality where $\Sigma_1$ is chosen to be the initial hypersurface with smooth and compactly supported spinor initial values $g^{\pm}$ on this hypersurface. We proceed as in \cite[Lem.2.6]{BaerStroh} and express $v,w$ with the solution operators from \Cref{cauchynormhyphom}:
\begin{eqnarray}
Q^{E}(t_2,t_1)&=&\rest{\Sigma_2}\circ D^{E}_{-}\circ \mathcal{G}^{-}(t_1)\circ (\upbeta\otimes \Iop{E\vert_{\Sigma_1}}) \label{QasFIO} \\  
\text{and}\quad\tilde{Q}^{E}(t_2,t_1)&=&\rest{\Sigma_2}\circ D^{E}_{+}\circ \mathcal{G}^{+}(t_1)\circ (\upbeta\otimes \Iop{E\vert_{\Sigma_1}}) \quad. \label{QnegasFIO}
\end{eqnarray}
In order to show that \clef{QasFIO} and \clef{QnegasFIO} are well defined on non-compact manifolds, the compositions of the two canonical relations $\mathsf{C}(\inclus^\ast)$ from \clef{relrest} with $\mathcal{C}$ and of the operators have to be well defined: 
\begin{itemize}
\item[(1)] $\rest{\Sigma_2}$ is properly supported;
\item[(2)] $\mathcal{G}^{\pm}(t_1)$ are properly supported;
\item[(3)] $D^{E}_{\pm}\circ \mathcal{G}^{\pm}(t_1) \in \FIO{-1/4}_{\mathsf{prop}}(M;\mathcal{C}';\Hom(\spinb^{\pm}_{E}(\Sigma_1),\spinb^{\pm}_{E}(M)))$; 
\item[(4)] $\mathsf{C}(\inclus^\ast)\circ \mathcal{C}$ transversal and proper.
\end{itemize}
\vspace{0.25cm}
(1) The statement follows from \Cref{correstfiopropp}. \\
\\ 
(2) Properly supportness follows from the finite propagation speed property of normally hyperbolic operators, explained in \cite[Thm.2]{AndBaer} for smooth and in \cite[Rem.16]{BaerWafo} for Sobolev regularity.
The solution $v$ in \clef{homcauchysol1} satisfies $\supp{v}\subset\Jlight{}(\supp{g^{-}})$ with $\supp{g^{-}}$ compact. Since $\Jlight{}(\supp{g^{-}})$ is itself spatially compact and $M$ temporal compact, the support of the solution $v$ is contained inside a compact set. The representation $v=\mathcal{G}^{-}(t_1)g^{-}$ with the solution operator then implies that the solution operator maps compactly supported sections on $\Sigma_1$ to sections with compact support on $M$. Suppose on the other hand that $\tilde{v}$ is a solution of the first equation in \clef{homcauchysol1} and has compact support $K\in M$. The past domain of dependence $\domdep{-}(K)\subset \Jlight{-}(K)$ contains all causal curves which contribute to the solution $\tilde{v}$ with support in $K$ via bicharacteristic transport. The Cauchy developements are closed and thus spatially compact. Since the dual of $\mathcal{G}^{-}(t_1)$ acts like a restriction to the initial hypersurface $\Sigma_1$, the action of the dual operator on $\tilde{v}$ has support inside the compact set $\domdep{-}(K)\cap\Sigma_1$ in $\Sigma_1$. Thus, also the dual operator maps compactly supported sections on $M$ to sections with compact support on $\Sigma_1$ and $\mathcal{G}^{-}(t_1)$ becomes properly supported in the manner of \clef{propsupp}. The argument works equally well for $\mathcal{G}^{+}(t_1)$. Moreover, the first compositions from the right in \clef{QasFIO} and \clef{QnegasFIO} are well-defined and a properly supported Fourier integral operator of order 0 with the canonical relation $\mathcal{C}$ from $\mathcal{G}^{\pm}(t_1)$ because $(\upbeta\otimes\Iop{E}\vert_{\Sigma_1})$ is a properly supported operator of order zero with canonical relation $N^\ast \mathrm{diag}(\Sigma_1)$.\\
\\ 
(3) Since differential operators on $M$ can be interpreted as FIO from $M$ to $M$ (see (3) in \Cref{fiopropalg}), the composition $N^\ast \mathrm{diag}(M)\circ \mathcal{C}$ is proper and transversal and results in $\mathcal{C}$. Because $D^E_{\pm}$ and $\mathcal{G}^{\pm}(t_1)$ are properly supported, part (2) from \Cref{fiopropalg} implies that $(D^{E}_{\mp}\circ \mathcal{G}^{\mp}(t_1))$ is an element in $\FIO{-1/4}_{\mathsf{prop}}(\Sigma,M;\mathcal{C}';\Hom(\spinb^{\mp}_{E}(\Sigma_1),\spinb^{\pm}_{E}(M)))$. 
%
\newpage
(4) The construction of the solution operators as FIO has been already done in such a way that the canonical relation $\mathsf{C}(\inclus^\ast)\circ\mathcal{C}$ as composition is transversal and proper. We refer to the explainations in \Cref{chap:solini} and in particular \cite[Chap.5]{duistfio} for details. The Dirac operators $D^{E}_{\pm}$ do not affect the argument because $N^\ast \mathrm{diag}(M)\circ \mathcal{C}=\mathcal{C}$ as explained in (3). \\
\\
\Cref{fiopropalg} (2) thus implies that the compositions \clef{QasFIO} and \clef{QnegasFIO} are indeed well-defined properly supported Fourier integral operators of order $0$. The composition $\mathsf{C}_{1\rightarrow 2}:=\mathsf{C}(\inclus^\ast)\circ\mathcal{C}$ of the canonical relations are given in accordance to \cite{BaerStroh} and is related to the lightlike cogeodesic flows in future and past direction from $\Sigma_1$ to $\Sigma_2$. As these are Hamiltonian flows, the canonical relations turn out to be graphs of canonical transformations. 
As in \cite[Lem.2.6]{BaerStroh} this observation allows to calculate the principal symbol of $Q$ by multiplying the principal symbols of each occuring operator according to \clef{fioprincsymbsymplecto} (recall \cite[Thm.4.2.2/3]{hoermfio} for the details). Analogous results can be obtained for $\tilde{Q}$ as well as $Q^E$ and $\tilde{Q}^E$ where for the latter two operators the parallel transport in the twisting bundle decomposes into a tensor product of parallel transports. As $H^s_\comp$ is the completion of $C^\infty_\comp$ and all wave evolution operators are properly supported, the construction can be extended to $H^s_\comp$ and $H^s_\loc$ (see \Cref{fiopropreg} (3)).
\end{proof} 

\appendix

\section{Fourier integral operators}\label{chap:Fourier}

We recall Fourier integral operators (FIO) as a class of operators between sections on possibly different manifolds which contains differential and pseudo-differential operators. We will only focus on their global properties as we will only use them in the following analysis. The presented content is based on the papers of Hörmander and Hörmander-Duistermaat (\cite{hoermfio} and \cite{hoermduistfio}) as well as the textbooks references \cite{duistfio}, \cite{hoerm4}. Further supporting notes are taken from \cite[Sec.2.1]{ornstroh} and \cite[App.B]{ornstroh}. We take the manifolds $M,N$ and the vector bundles $E$ and $F$ as introduced in the first subsection of \Cref{chap:Back}.\\
\\
Let $\mathsf{\Lambda} \subset \dot{T}^\ast(N\times M)$ be a closed conic Lagrangian submanifold with a symplectic form on $T^\ast(N\times M)$. A \textit{homogeneous canonical relation} from $\dot{T}^\ast M$ to $\dot{T}^\ast N$ is a closed conic Lagrangian submanifold $\mathsf{C}$ in $\dot{T}^\ast(N\times M)$, given by
\begin{equation*}
\mathsf{C}=\Bigl\{(x,\xi,y,\eta) \in \dot{T}^\ast M \times \dot{T}^\ast N \Big\vert (x,\xi,y,-\eta)\in\mathsf{\Lambda}\Bigr\}=:\Lambda'\quad.
\end{equation*}
A \textit{Lagrangian distribution} of order $r\in \R$ associated to the Lagrangian submanifold $\mathsf{\Lambda}=\mathsf{C}'$ can be represented as sum of locally finite supported oscillatory integrals in a coordinate neighbourhood of $M$ with non-degenerate phase functions. We designate $I^r(M;\mathsf{\Lambda},E)$ as the set of Lagrangian distributions of order $r\in \R$, associated to the Lagrangian submanifold $\mathsf{\Lambda}$. 
\begin{defi}[cf. Definition 25.2.1 in \cite{hoerm4}]\label{defifourier}
Given $E\rightarrow M$ and $F\rightarrow N$ and a closed conic Lagrangian submanifold $\mathsf{\Lambda} \subset \dot{T}^\ast(M\times N)$ with corresponding homogeneous canonical relation $\mathsf{C}$ from $\dot{T}^\ast N$ to $\dot{T}^\ast M$; a \textit{Fourier integral operator of order} $r\in \R$ is an operator $P:C^\infty_\comp(M,E)\rightarrow C^{-\infty}(N,F)$ with Schwartz kernel $K\in I^r(N\times M;\mathsf{\Lambda},\Hom(E,F))$.
\end{defi}
We will designate the space of those operators with $\FIO{r}(M,N;\mathsf{C}';\Hom(E,F))$. The local description of the kernels can be recapitulated e.g. in \cite[Sec.2.1]{ornstroh}. \\
\\
We set $\mathsf{\Lambda}^{-1}$ as inverse of the closed conic Lagrangian submanifold and $\mathsf{C}^{-1}$ denotes the corresponding inverse canonical relation which is itself a canonical relation from $\dot{T}^\ast N$ and $\dot{T}^\ast M$. Let $W$ be another manifold; the composition of two homogeneous canonical relations $\mathsf{C}_1$ from $\dot{T}^\ast N$ to $\dot{T}^\ast W$ and $\mathsf{C}_2$ from $\dot{T}^\ast M$ to $\dot{T}^\ast N$ is
\begin{multline*}
\mathsf{C}_1\circ\mathsf{C}_2:=\Bigl\lbrace (x,\xi,z,\zeta)\in \dot{T}^\ast M \times \dot{T}^\ast W\,\Big\vert\,\exists\,(y,\eta)\in \dot{T}^\ast N\,:\,(x,\xi,y,\eta)\in \mathsf{C}_2 \\
\quad\text{and}\quad (y,\eta,z,\zeta)\in \mathsf{C}_1\Bigr\rbrace \quad.
\end{multline*} 
The composition is called \textit{transversal} if
\begin{equation*}
T_p\widetilde{\mathsf{C}}=T_p(\mathsf{C}_1\times\mathsf{C}_2)+T_p(T^\ast M \times \mathrm{diag}(T^\ast N)\times T^\ast W)
\end{equation*} 
for all points $p$ in $\widetilde{\mathsf{C}}$. The composition is called \textit{proper} if the projection $\widetilde{\mathsf{C}}\rightarrow\dot{T}^\ast(M\times W)$ is a proper map. A special situation arises for $m=n$. A homogeneous canonical relation from $\dot{T}^\ast M$ to $\dot{T}^\ast N$ will be called \textit{local canonical graph} if both projections on $\dot{T}^\ast M$ and $\dot{T}^\ast N$ are local diffeomorphisms. The homogeneous canonical relation is locally the graph of a canonical transformation and a symplectic manifold on its own right. It is called \textit{bijective} if in addition $\mathsf{C}^{-1}$ is a local canonical graph.\\
\\
We first collect some algebraic properties of Fourier integral operators which are proven in and taken from \cite[Sec.25.2]{hoerm4} and \cite[Chap.4]{hoermfio}, supported with additional details from \cite[Sec.4.1]{ornstroh}.
\begin{lem}\label{fiopropalg}
Given $E\rightarrow M$ and $F\rightarrow N$, a vector bundle $G\rightarrow W$ and homogeneous canonical relations $\mathsf{C}=\mathsf{C}_1$ from $\dot{T}^\ast M$ to $\dot{T}^\ast N$ and $\mathsf{C}_2$ from $\dot{T}^\ast N$ to $\dot{T}^\ast W$; the following properties hold for all $r,s \in \R$,
\begin{itemize}
\item[(1)] (adjoint FIO) if $A \in \FIO{s}(M,N;\mathsf{C}';\Hom(E,F))$, then the adjoint/dual 
satisfies
\begin{equation*}
A^\dagger \in \FIO{r}(N,M;(\mathsf{C}^{-1})';\Hom(\overline{F}^\ast,\overline{E}^\ast))\quad ;
\end{equation*} 
\item[(2)] (composition) given two operators 
\begin{equation*}
\begin{split}
A_2 &\in \FIO{r}(M,N;\mathsf{C}_2;\Hom(E,F)) \\
A_1 & \in \FIO{s}(N,W;\mathsf{C}_1;\Hom(F,G)) \quad ;
\end{split}
\end{equation*}
if the operators are properly supported and the composition $\mathsf{C}_1\circ \mathsf{C}_2$ is transversal and proper, then $A_1\circ A_2 \in \FIO{r+s}(M,W;(\mathsf{C}_1\circ\mathsf{C}_2)';\Hom(E,G))$.
\item[(3)] let $N=M$, then $\ydo{r}{}(M,\Hom(E,F))\subset \FIO{r}(M,M;(N^\ast\mathrm{diag}(M))';\Hom(E,F))$;
\item[(4)] $A \in \FIO{-\infty}(M,N;\mathsf{C}';\Hom(E,F))$ if and only if $A$ is smoothing.
\item[(5)] $(A^\dagger\circ A) \in \ydo{2m}{}(M,\End(E))$ for $A$ as in (1).
\end{itemize}  
\end{lem} 
Property (5) is a consequence of (1), (2), and (3) with $\mathsf{C}^{-1}\circ\mathsf{C}=N^\ast\mathrm{diag}(M)$. We denote the set of properly supported Fourier integral operators of order $r$ with $\FIO{r}_{\mathsf{prop}}$. If the homogeneous canonical relation is a local canonical graph, then one can show the following two regularity properties of Fourier integral operators which are also based on results, proven in \cite{hoerm4} and \cite{hoermfio}. 
\begin{lem}\label{fiopropreg}
If $\mathsf{C}$ is a local canonical graph from $\dot{T}^\ast M$ to $\dot{T}^\ast N$ and $A \in \FIO{r}(M,N;\mathsf{C}';\Hom(E,F))$, then the following holds:
\begin{itemize}
\item[(1)] for $r=0$, $A$ becomes a continuous map from $L^2_{\comp}(M,E)$ to $L^2_{\loc}(N,F)$; if
\begin{equation*}
\sup_{(x,y)\in K}\norm{{\sigma}_0(A)(x,\xi;y,\eta)}{\Hom(E,F)} \quad \rightarrow \quad 0
\end{equation*}
for $\absval{(\xi,\eta)}\rightarrow \infty$ for all $K \Subset M \times N$, then it maps as compact operator between $L^2(M,E)$ to $L^2(N,F)$.
\item[(2)] the operator $A$ maps continuously from $H^s_{\comp}(M,E)$ to $H^{s-r}_{\loc}(N,F)$ for all $s\in \R$;
\item[(3)] if $A$ is properly supported, then $A:H^s_\comp(M,E)\rightarrow H^{s-r}_\comp(N,F)$ and\\ $A:H^s_\loc(M,E)\rightarrow H^{s-r}_\loc(N,F)$ are continuous linear maps for all $s\in \R$.
\end{itemize}  
\end{lem} 
The principal symbol of $A\in \FIO{r}(M,N;\mathsf{C}';\Hom(E,F))$ is an element in the equivalence class of symbols of order $(r+(n+m)/4)$ modulo symbols of order $(r+(n+m)/4-1)$ which are sections of the bundle $M_{\mathsf{C}}\otimes\pi_{\mathsf{C}}^\ast(\Hom(E,F))\rightarrow \mathsf{C}$; here $\pi_{\mathsf{C}}$ is the bundle projection $T^\ast(N\times M)\supset \mathsf{C}\rightarrow N\times M$ and $M_{\mathsf{C}}$ is the \textit{Keller-Maslov bundle} which is a complex trivial line bundle describing the invariance of the amplitude and the half-density of an oscillatory integral under the change of the phase function due to coordinate transformations. More informations about the geometric nature of the Keller-Maslov bundle are given in \cite[Sec.4.1]{duistfio} and \cite[Sec.3.3]{hoermfio}. \\
\\
The principal symbol of the composition of two Fourier integral operators with the required assumptions from \Cref{fiopropalg} (2) is a more involved combination of the single principal symbols; see \cite[Thm.25.2.3]{hoerm4} or \cite[Sec.3.2]{hoermfio}. The situation becomes much easier if both homogeneous canonical relations of the composing Fourier integral operators are local canonical graphs. The composition of canonical relations becomes itself a graph of a composition of two symplectomorphisms and the principal symbol of the composition becomes 
\begin{equation}\label{fioprincsymbsymplecto}
{\sigma}_{r+s}(A_1\circ A_2)={\sigma}_{s}(A_1)\circ {\sigma}_{r}(A_2) \quad.
\end{equation}
We refer to the literature and the content of \cite[Sec.2.1/App.B]{ornstroh} for details and the local expression of the principal symbol. The notion of a principal symbol allows to consider exact sequences between properly supported Fourier integral operators. 
\begin{lem}\label{fiopropsymb}
For any operator $A \in \FIO{r}_{\mathsf{prop}}(M,N;\mathsf{C}';\Hom(E,F))$ exists an exact sequence
\begin{multline*}
0\rightarrow \FIO{r-1}_{\mathsf{prop}}(M,N;\mathsf{C}';\Hom(E,F))\hookrightarrow\FIO{r}_{\mathsf{prop}}(M,N;\mathsf{C}';\Hom(E,F)) \\
\stackrel{\sigma_r}{\rightarrow} C^\infty(\mathsf{C},M_{\mathsf{C}}\otimes\pi_{\mathsf{C}}^\ast\Hom(E,F)) \rightarrow 0
\end{multline*}
\end{lem} 
At the end of this appendix, we take a closer look on one special example of a FIO. Let $E\rightarrow M$ be as before and now $N$ is an embedded submanifold of codimension $k$ in $M$ with inclusion $\inclus\,:\,N\,\hookrightarrow\,M$. The pullback of the embedding defines the \textit{restriction operator} 
\begin{equation}\label{restop}
\rest{N}:={\inclus}^\ast \,:\, C^\infty(M,E)\quad\rightarrow\quad C^\infty(N,E\vert_N) \quad
\end{equation}
which assigns each function its trace on the submanifold. The adjoint/dual of $\rest{N}$ with respect to the dual pairing $C^{-\infty}_\comp \times C^\infty\rightarrow \C$ is the \textit{corestriction operator}
\begin{equation}\label{corestop}
\rest{N}^\dagger \,:\, C^{-\infty}_\comp(N,E\vert_N)\quad\rightarrow\quad C^{-\infty}_\comp(M,N) \quad.
\end{equation}
As pushforward is the dual operation to pullback, one also writes $\inclus_\ast$ for $\rest{N}^\dagger$. We outline the important properties.
\begin{prop}\label{restcorestfio}
Let $N$ be an embedded smooth submanifold of a smooth manifold $M$ with codimension $k$ and inclusion map $\inclus\,:\,N\,\hookrightarrow\,M$ as well as a vector bundle $E\rightarrow M$; then the following holds.
\begin{itemize}
\item[(1)] 
\item[\quad\quad\quad\quad a)] $\inclus^\ast \in \FIO{k/4}(M,N;\mathsf{C}'(\inclus^\ast);\Hom(E,E\vert_N))$ with homogenous canonical relation
\begin{equation} \label{relrest}
\mathsf{C}(\inclus^\ast):=\SET{(y,\eta,x,\xi)\in \dot{T}^\ast(N\times M)\,\Big\vert\, (x,\xi) \in \dot{T}^\ast M \,:\, \inclus^\ast(x,\xi)=(y,\eta)}\,.
\end{equation}
\item[\quad\quad\quad\quad b)] $\inclus_\ast \in \FIO{k/4}(N,M;\mathsf{C}'(\inclus_\ast);\Hom(E\vert_N,E))$ with homogenous canonical relation
$$\mathsf{C}(\inclus_\ast):=N^\ast \mathrm{graph}(\inclus)=\mathsf{C}(\inclus^\ast)^{-1}\quad. $$
\item[(2)] if $N$ is a closed subset in $M$, then both operators \clef{restop} and \clef{corestop} are properly supported.  
\end{itemize}
\end{prop}
\begin{proof}
It has been shown in \cite[Sec.5.1]{duistfio} that the restriction operator is a Fourier integral operator of order $k/4$ with claimed homogeneous canonical relation in the scalar case. This carries over to the vector-valued case as this holds in any trivialisation of the vector bundle. The claim for the corestriction operator follows from \Cref{fiopropalg} (1) which concludes the proof for (1). \\
\\
The corestriction operator maps compactly supported distributions to compactly supported distributions by definition. Suppose $u\in C^\infty(M,E)$ has compact support in $K\subset M$. The pullback $\inclus^\ast u$ is well-defined, as $u$ has empty wave-front set, and has compact support in $\inclus(N)\cap K$ which is w.l.o.g. not empty, if otherwise, there is nothing to show. Since $N$ is a closed subset in $M$, the embedding is a closed map as well. The intersection of the compact support with the closed subset $\inclus(N)$ is a closed subset of $K$ and thus itself compact. Then the restriction maps compactly supported sections to compactly supported sections and due to \clef{propsupp} it is a properly supported operator in this situation. 
\end{proof}
If $M$ is a globally hyperbolic manifold, \Cref{restcorestfio} (3) is always satisfied because its Cauchy hypersurface $\Sigma$ is always closed. 
\begin{cor}\label{correstfiopropp}
\begin{equation*}
\inclus^\ast \in \FIO{1/4}_{\mathsf{prop}}(M,\Sigma;\mathsf{C}'(\inclus^\ast);\Hom(E,E\vert_\Sigma))\quad\text{\&}\quad \inclus_\ast \in \FIO{1/4}_{\mathsf{prop}}(\Sigma,M;\mathsf{C}'(\inclus_\ast);\Hom(E\vert_\Sigma,E))\,.
\end{equation*}
\end{cor}
%

\section{Solution operators for normally hyperbolic operators}\label{chap:solini}

Let $(M,\met)$ be a globally hyperbolic manifold with $\Sigma$ a Cauchy hypersurface and $E\rightarrow M$ a vector bundle. An operator $P\in \Diff{2}{}(M,\End(E))$ is called \textit{normally hyperbolic} if its principal symbol is determined by the metric: 
\begin{equation}\label{normhyp}
{\sigma}_2(P)(p,\xi)=\pm\met_p(\xi^\sharp,\xi^\sharp) \Iop{E} \quad.
\end{equation} 
The vanishing of the principal symbol corresponds to the vanishing of $\met_p(\xi^\sharp,\xi^\sharp)$ at each point $p \in M$ which is fulfilled for $\xi^\sharp$, being a lightlike vector at $x$. The bicharacteristic strips are determined by Hamilton's equations for the Hamilton function $1/2\met_p(\xi^\sharp,\xi^\sharp)$ which can be reduced to the geodesic equation on $M$. Thus, the bicharacteristic curves (projection of the bicharacteristic strip on $M$) are given by lightlike geodesics.\\
\\
Theorem A.1 in \cite{BaerStroh}, based on the more general treatment for Theorem 5.1.6 in \cite{duistfio} but for scalar-valued differential operators, is focused on solutions of the Cauchy problem
\begin{equation}\label{Cauchynormhyp}
\begin{split}
Pu & =0\quad \text{in} \quad M \\
\rest{\Sigma}u =g_0\quad&\text{and}\quad\rest{\Sigma}(\nabla_{\partial_t})u =g_1 
\end{split}
\end{equation}
for smooth and compactly supported initial data $g_0,g_1$ on the initial hypersurface in terms of solutions operators. We recall the statement.
\begin{theo}\label{cauchynormhyphom}
Let $M$ be a globally hyperbolic manifold with Cauchy hypersurface $\Sigma$ and $E\rightarrow M$ a vector bundle; the Cauchy problem \clef{Cauchynormhyp} for a normally hyperbolic operator $P\in \Diff{2}{}(M,\End(E))$ with initial hypersurface $\Sigma$ has a unique solution $u$ for every $g_0,g_1 \in C^\infty_\comp(\Sigma)$ such that $u =\sum_{j=0}^1 \mathcal{G}_j g_j$ where the solution operators $\mathcal{G}_j$ are continuous mappings from $C_\comp^\infty(\Sigma_0,E\vert_{\Sigma_0})$ to $C^\infty(M,E)$ with the properties
\begin{itemize}
\item[(1)] $\supp{\mathcal{G}_j} \subset \SET{(p,x)\in M \times \Sigma_0\,\vert\, x \in \Jlight{-}(p)\cap \Sigma_0}$ and
\item[(2)] $\mathcal{G}_j \in \FIO{-j-1/4}(\Sigma_0,M;\mathcal{C}';\Hom(E\vert_{\Sigma_0},E))$ with $\mathcal{C}$ as in \clef{canrelationnormhyp}.
\end{itemize} 
\end{theo} 
As mentioned above, the proof of \cite[Thm.5.6.1]{duistfio} carries over from scalar-valued operators to operators of real principal type as well as operators with scalar-valued principal symbols like normally hyperbolic operators. Global hyperbolicity in fact implies that all conditions for applying \cite[Thm.5.6.1]{duistfio} are satisfied. We refer to \cite[Prop.4.3/Prop.4.4]{radz} for the arguments as well as some additional explainations in \cite[App.A]{ODT}. The canonical relation for the solution operators is 
\begin{equation}\label{canrelationnormhyp}
\mathcal{C}:=\SET{(x,\xi,y,\eta) \in \dot{T}^\ast M \times \dot{T}^\ast \Sigma_0\,\Big\vert\, (x,\xi)\sim(y,\eta)}
\end{equation}
where $(x,\xi)\sim(y,\eta)$ means that there is a lightlike vector $\zeta \in T^\ast_y M$ such that the points $(x_0,\xi)$ and $(y,\zeta)$ are are connected by the same orbit of the null geodesic flow in $M$ with $\rest{\Sigma_0}^\ast \zeta = \eta$. In order to distinguish both directions, one considers $\xi$ to be a future- or past-directed lightlike vector such that $(x,\xi)$ and $(y,\zeta)$ are connected by a null geodesic from past to future or vice versa. This allows to decompose $\mathcal{C}$ into two connected components $\mathcal{C}^{\pm}$:
\begin{equation}\label{canrelationnormhypsplit}
\mathcal{C}^{+} = \SET{(x,\xi,y,\eta) \in \mathcal{C}\,\vert\, \xi \vartriangleright 0}\quad\text{and}\quad \mathcal{C}^{-} = \SET{(x,\xi,y,\eta) \in \mathcal{C}\,\vert\, \xi \vartriangleleft 0} 
\end{equation}
where $\xi\vartriangleright 0$ means that $(+\xi)$ is a future-directed lightlike covector and $\xi\vartriangleleft 0$ means that $\xi$ past-directed lightlike covector. This decomposition has been used in \cite[Appendix]{BaerStroh} to trivialise the Keller-Maslov line bundle over the canonical relation. 

\end{document}

%% file: footnotes.tex
\bnotecontent{f1}{Since each level set is an embedded hypersurface, the inclusion maps $\inclus_t:\, \Sigma_t\,\hookrightarrow\,M$ are smooth embeddings, i.e. $\inclus_t$ are diffeomorphisms between $\Sigma_t$ and $\inclus_t(\Sigma_t)$.}

\bnotecontent{f2}{The term \textit{inward/outward pointing normal vector field} can be rephrased in the Lorentzian setting: if the hypersurface is spacelike, one fixes one of the two timelike normal directions as normal field (here past-directed) and calls it outward pointing. The other time-oriented normal vector is thus the inward pointing normal vector (here future-directed).
}
\bnotecontent{f1c}{Here and in the following if the manifold is clear from the context or from the bundle $E$, we write $C^\infty(E)$ instead of $C^\infty(M,E)$ and the same for other sections.}

\bnotecontent{f3}{Here and in the following we choose the sesquilinearity in such a way that the first entry is anti-linear.}

\bnotecontent{f4}{Another way of introducing a $\Psi$DO is presented in the appendix \ref{chap:app2} as a Fourier integral operator.}

\bnotecontent{f5}{A more general version of Seeleys result in 1964 can be found in \cite{Melbook} as Theorem and Lemma 1.4.1. in chapter 1.}

\bnotecontent{f11}{In comparison, $L^2_\comp(M,E)$ denotes those sections in $L^2_\loc(M,E)$, having compact support. In order to stress a fixed compact support in $K \Subset M$, I write $$L^2_{K,\loc}(M,E):=\SET{u \in L^2_\loc(M,E)\,\vert\, \supp{u} \subset K}\quad. $$ Then one can define $L^2_\comp(M,E)$ as union of $L^2_{K,\loc}(M,E)$ over all $K$ compact in $M$. }

\bnotecontent{f12}{This follows from the essentially self-adjointness of the identity map and the Bochner Laplacian $(\Nabla{E}{})^\ast\Nabla{E}{}$, see \cite{BraMilShub} or \cite{yoshida}. It holds true for any compact manifold since they are complete by the Hopf-Rinow theorem.}

\bnotecontent{f13}{If one has $k+\alpha < s-\frac{\dim(M)}{2}$ with $\alpha \in \intervallo{}{0}{1}$, one has even more a continuous embedding into the Hölder spaces $C^{k,\alpha}(M,E)$.}

\bnotecontent{f27}{In comparison the \textit{domain of dependence} (also known as \textit{"causal diamond"}) is defined as $\domdep{}(A)=\domdep{+}(A)\cup \domdep{-}(A)$ for $A \subset M$ such that no two contained points can be connected by a timelike curve (i.e. $A$ is achronal) where
\begin{equation*}
\domdep{\pm}(A):=\SET{p \in M\,\vert\,\text{every past}(+)\setminus\,\text{future}(-)\,\text{inextendible causal curve through}\,p\,\text{meets}\,A}\quad.
\end{equation*}
In particular $A \subset \domdep{\pm}(A)$, but also $\domdep{\pm}(A)\subset\Jlight{\pm}(A)$. }

\bnotecontent{fc1}{Later on $\mathcal{M}=\Sigma$ or $\Sigma_t$ for a $t \in \timef(M)$.}

\bnotecontent{f6}{Thus, it is an oriented manifold together with a $\group{Spin}_0(1,n)$-principal bundle where 
\begin{equation*}
\group{Spin}_0(1,n):=\SET{e_1,...,e_{2k}\in \group{Cl}^0_{1,n}\,\vert\,v_j \in \R^{n+1}\, , \, \idscal{1}{}{v_j}{v_j}=\pm 1 \quad \text{and} \quad \prod_{j=0}^{2k}\idscal{1}{}{v_j}{v_j}=1}
\end{equation*}
with $\group{Cl}^0_{1,n}$ as even elements of the Clifford-algebra w.r.t. the symmetric bilinearform $\idscal{1}{}{\cdot}{\cdot}$ is the connected component of the identity in the spin group; see \cite{BaerGauMor} and related literature for more details.}

\bnotecontent{f10}{Sections of $\spinb^{\pm}(\Sigma_t)$ for each hypersurface can be interpreted as spinor fields on the hypersurface with positive or negative chirality, even though I used it as $\spinb^{\pm}(\Sigma_{t})=\spinb^{\pm}(M)\vert_{\Sigma_t}=\spinb(\Sigma_t)$ for both signs. But since I won't apply a chirality decomposition on the hypersurfaces, this shouldn't bother here.}

\bnotecontent{f14}{See \cref{remsdiracop} (iii).}

\bnotecontent{f28}{Manifolds with bounded geometry are complete which follows from the injectitivity radius condition; see \cite[Def.1.1]{shubinspec} and following. Thus, the results from the former sections can be applied.}

\bnotecontent{f7}{This circumstand might lead to an obstruction in using more general boundary conditions as e.g. given in \cite{§4.2}{BaerHan} where the multiplicities induce an additional term in the index formula for the compact case if one is using this kind of generalized Atiyah-Patodi-Singer boundary conditions.}

\bnotecontent{f8}{The same procedure will be applied for the flip metric of $\met$ in order to compute the geometric index formula.}

\bnotecontent{f9}{It is also possible to achieve that every leave $\mathcal{T}=const$ in a globally hyperbolic manifold is complete and a Cauchy hypersurface by modifying the lapse function $N$ to be bounded on every leave; this shown in \cite[Prop.15]{mueller}.}

\bnotecontent{f15}{See appendix \ref{chap:app2} for the notations.}

\bnotecontent{f20}{This can be seen in Definition \ref{defevop} since the Dirac wave evolution operator maps compactly supported Sobolev sections to compactly supported Sobolev sections and is a topological isomorphism, so its adjoint does as well.}

\bnotecontent{f16}{The boundness on $\kernel{B}$ is inherited from the boundness on $\mathscr{H}$.}

\bnotecontent{f17}{Equivalently, the orbit of each point in $M$ is $M$ itself.}

\bnotecontent{f18}{Another reference for the Dirac operator on a Riemannian $\Upgamma$-manifold is \cite[Prop.3.1]{atiyahellvn}.}

\bnotecontent{f19}{$B\ydo{\ast}{(\mathsf{cl})}$ is closed under composition.}

\bnotecontent{f29}{Shubin introduced in \cite{shub} this class to consist of those pseudo-differential operators which differ from a properly supported pseudo-differential operator in a smoothing operator in $\ydo{-\infty}{\Upgamma}$. But the definition given here coincides with the manner, how these operators have been actually used in this reference.}

\bnotecontent{f21}{See e.g. \cite[Thm.3.2]{Joshi} for the more general statement that the tensor product of two local Sobolev sections is again a local section.}

\bnotecontent{f22}{Since $\upnu$ is chosen to be globally past-directed, the temporal projection of $\zeta_{+}$ is anti-parallel to $\upnu$ whereas the projection of $\zeta_{-}$ in temporal direction is parallel; the same holds for $\varsigma_{\pm}$.}

\bnotecontent{f23}{Projections on a Banach space are bounded iff their kernels and range are closed subspaces of this Banach space.}

\bnotecontent{f24}{See \cite[App.A]{benroy} for the unitary identification of $\mathscr{K}_\Upgamma(\mathscr{H})$ and $\mathscr{N}_{r}(\Upgamma)\otimes\mathscr{K}(\mathcal{H})$.}

\bnotecontent{f25}{Here and in the following of this proof every appearing topological isomorphism between Hilbert $\Upgamma$-modules can be considered as unitary isomorphism; see \cite[Prop.2.16]{shub}. Thus, every appearing isomorphism between Hilbert $\Upgamma$-modules implies that their $\Upgamma$-dimensions coincide.}

\bnotecontent{f26}{The mixed terms can be estimated with the polarization identity of the form
\begin{equation*}
4\Rep{\dscal{1}{\mathcal{H}}{x}{y}}=\norm{x+y}{\mathcal{H}}^2-\norm{x-y}{\mathcal{H}}^2 \leq \norm{x+y}{\mathcal{H}}^2+\norm{x-y}{\mathcal{H}}^2 = 2\norm{x}{\mathcal{H}}^2+2\norm{y}{\mathcal{H}}^2
\end{equation*}
for $x,y\in \mathcal{H}$ Hilbert space. This explains the factor 2 in the third line which will be restored into the constant in front.}

\bnotecontent{f30}{We can check that the closedness of $S_t$ implies the closedness of $\hat{S}_t$ for each $t$ fixed, due to the fact that the wave evolution operator is bounded. The condition $\kernel{\hat{S_t}^\ast \pm \Imag \Iop{\mathscr{H}_2(a)}}=\SET{0}$ is implied by the preassumption $\kernel{S_t^\ast\pm\Imag \Iop{\mathscr{H}_2(t)}}=\SET{0}$ since $Q(t,a)$ and $Q(a,t)$ are isomorphisms.}

\bnotecontent{f31}{This becomes clear if we apply the chain and product rule for operators with respect to a scalar variable dependence: Let $u \in \dom{}{A^2}\subset \dom{}{A}$ and $A$ time dependent with derivative $\dot{A}$, then
\begin{equation*}
2A\dot{A}u=\frac{\differ }{\differ t}(A^2u)=\dot{A}Au+A\dot{A}u \quad \leftrightarrow\quad A\dot{A}u=\dot{A}Au \quad .
\end{equation*}
}

\bnotecontent{f32}{$f^\ast \mathtt{c}(E,\nabla)=\mathtt{c}(f^\ast E,f^\ast \nabla)$ for any continuous function $f$. The curvature satisfies $\Omega(f^\ast\nabla)(X,Y)f^\ast u=f^\ast\left[\Omega(f_\ast X,f_\ast Y)u\right]$ for $X,Y$ vector fields on $M'$ and $u$ as in the text.}

\bnotecontent{af1}{See \cite{nazaetal} and \cite{sipa} for more details.}

\bnotecontent{af2}{Certain wavefront set conditions have to hold for both cases in addition and for each case a transversality condition: Let 
\begin{equation*}
N^\ast\mathrm{graph}(f):=\SET{(x,\xi,y,\eta)\in T^\ast(X\times Y)\,\vert\, y=f(x) \quad\text{and}\quad \xi\in(\differ f\vert_x)^{\dagger}(\SET{\eta})}\subset T^\ast X \times T^\ast Y
\end{equation*}
denote the conormal bundle of the graph of $f$, then $\mathsf{s}^\ast(N^\ast\mathrm{graph}(f))\pitchfork T^\ast X \times \mathsf{\Lambda}_Y$ for the pullback and \\
$\mathsf{s}^\ast(N^\ast\mathrm{graph}(f))\pitchfork \mathsf{\Lambda}_X \times T^\ast Y$ for the pushforward. $\mathsf{\Lambda}_X$ and $\mathsf{\Lambda}_Y$ are Lagrangian submanifolds of the mapped Lagrangian distributions and $(\differ f\vert_p)^{\dagger}\,:\,T^\ast_{f(p)}Y\,\rightarrow\,T^\ast_p X$ is the adjoint map of the pushforward which is a priori defined for $f$ being a diffeomorphism, but used here for the preimage.}

\bnotecontent{F1}{Dirac operators and the Gauss-Bonnet operator are examples of those geometric operators.}

\bnotecontent{F2}{Calculating the $\Upgamma$-(co-)dimensions shows
\begin{equation*}
\begin{split}
\dim_\Upgamma\left(\range{P_{> 0}}\cap\left(\range{P_{\geq 0}}\right)^\perp\right)&= \dim_\Upgamma\left(\range{P_{> 0}}\cap\range{P_{< 0}}\right)=\dim_\Upgamma(\range{P_{\emptyset}})=0 \\
-\Index_\Upgamma(\overline{P}^{>0}_{\geq 0}(\tau))&=\codim_\Upgamma\left(\range{P_{> 0}}\cap\range{P_{\geq 0}} \right)= \dim_\Upgamma\left(\quotspace{\range{P_{\geq 0}}}{\range{P_{\geq 0}}\cap\range{P_{>0}}}\right)\\
&=\dim_\Upgamma\range{P_{0}}=\dim_\Upgamma\kernel{A_\tau}< \infty \quad.
\end{split}
\end{equation*} 
}

\bnotecontent{F3}{E.g. the spin-Dirac, signature- or Gauss-Bonnet operator.}

\bnotecontent{F4}{In comparison: a \textit{geometric product space} is the product space $\R\times \Sigma$ with product metric \clef{metricspinorbundle} where the hypersurface metric $\met_\Sigma$ is fixed for all $t\in \R$.}

\bnotecontent{F5}{The group $\group{SO}$ in the subscript of the frame bundle is an abbreviation for $\group{SO}(r,s-1)$.}

\bnotecontent{F6}{It is common to discern the orbits spaces for left and right actions with $\group{G}/M$ respectively $M/\group{G}$. We don't distinguish between these two designations since we have restricted ourselves to left actions.}

\bnotecontent{F7}{Hereby we mean from now on a $\Upgamma$-manifold with smooth $\Upgamma$-invariant density with a possible $\Upgamma$-vector bundle, equipped with a $\Upgamma$-invariant bundle metric.}

\bnotecontent{F8}{We also suggest the introduction of \cite{Taylor1976} for more details as well as some analysis about the types if one replaces $\Upgamma$ with a suitable locally compact, but not necessarily discrete groups.}

\bnotecontent{F9}{The term unbounded means that the operator is not necessarily bounded} 

\bnotecontent{F10}{Other authors use the terminology $\Upgamma$-equivariant operator to stress that the operator intertwines the two left translation representations.}

\bnotecontent{F11}{See e.g. \cite[Prop.4.1]{vaill}, \cite[Lem.6.5]{vaill} or \cite[Thm.6.21]{schickl2} and \cite[App.A]{benroy} for the identification.}

\bnotecontent{F12}{See \cite[Thm.B]{PMIHES1969} or one recalls \cite[Prop.1]{phillips1996}.}

\bnotecontent{F13}{For $p,q\in [1,\infty)$ we call $(p,q)$ a \textit{conjugated (number) pair} if $1=1/p+1/q$.}

\bnotecontent{F14}{This is the \textit{east-coast convention} which we are going to use in this thesis. Another common way to define a Lorentzian metric is the choice $s=n-1$ which is the so-called \textit{west-coast convention}.}

\bnotecontent{F15}{In other contexts where we need to distinguish between spacelike and timelike properties, we also use the synonyms spatial respectively temporal.}

\bnotecontent{F16}{The case of such coverings is proven in \cite{Cheeger1990ChoppingRM}.}

\bnotecontent{F17}{Later, we will clarify that the $\Upgamma$-modules needs to be free or even more projective.}

\bnotecontent{F18}{This concept of Breuer- or $L^2$-Fredholmness in a von Neumann setting will be explained in \cref{chap:galois-sec:vneumann3} where we will introduce the terminology $\Upgamma$-Fredholmness.}

\bnotecontent{F19}{The time coordinate is treated as complex variable, constraint to the real axis. The Wick rotation rotates the real time axis to a imaginary time axis. In this way, the Lorentzian spacetime can be interpreted as Euclidean spacetime with an imaginary time coordinate.}

\bnotecontent{F20}{\protect\label{tempcompfn}\textit{Temporal compactness} means that the time domain $\timef(M)$ of the globally hyperbolic manifold $M$ is a compact time interval, i.e. there exists $t_1,t_2 \in \R$ such that $\timef(M)=[t_1,t_2]$. We will also use this terminology to express that we restrict the possibly non-compact time domain of $M$ to any compact time interval $[t_1,t_2]$. Thus, any restriction $M\vert_{[t_1,t_2]}$ becomes temporal compact in the original sense.}

\bnotecontent{F21}{\protect\label{secopfn}Some authors consider the sector to be contained in the resolvent set $\uprho(A)$ which is just a change of the point of view.}

\bnotecontent{F22}{As we will only consider connections on a vector bundle, we will parallely use the term connection for covariant derivative.}

\bnotecontent{F23}{Next to orientability (i.e the vanishing of the first Stiefel-Whitney class of $TM$) also its second Stiefel-Whitney class has to vanish. 
}

\bnotecontent{F24}{Since $M$ is time- and space-oriented, its tangent bundle obeys the decomposition \clef{decomptangentpseudoriem}. If the subbundles $T^{\pm}M$ admit themselves spin structures, $TM$ has vanishing second Stiefel-Whitney class. For more details, see \cite[Sec.2.1]{baum1981spin}.}

\bnotecontent{F25}{In order to apply one of these functional calculi, one needs to make sure that the operators of interest are self-adjoint with respect to a positive definite Hermitean bilinear form on the vector bundle. Moreover, the zero eigenvalue has to lie outside the essential spectrum.}

\bnotecontent{F26}{No two points in $\Sigma$ can be connected with a timelike curve.}

\bnotecontent{F27}{To be more precise, $\inclus_\ast Y$ is a section of the pullback bundle $\inclus^\ast(TM)$ along $\Sigma$.}

\bnotecontent{F28}{We will use this convention for any other occuring space of operators, acting between sections on one and the same manifold.}

\bnotecontent{F29}{$K_1$ has to be chosen in such a way that its boundary $\bound K_1$ is totally geodesic with normal vector field $\mathfrak{n}$ such that all normal derivatives of $\curvcon(X,\mathfrak{n})\mathfrak{n}$ vanish; see \cite{Mori} for details}

\bnotecontent{F30}{$\group{Mat}(n,\C)=\group{Mat}_{n\times n}(\C)$ is the algebra of complex $(n\times n)$-matrices.}

\bnotecontent{F31}{$(r,s)=(0,n)$ implies a negative definite metric which becomes Riemannian after rescaling with $(-1)$; we only refer to $(r,s)=(n,0)$ as the Riemannian case as the other possibility differs in a global sign.}

\bnotecontent{F32}{This is still justified in the non-compact case because the Cauchy hypersurfaces are assumed to be complete.}

\bnotecontent{F33}{\protect\label{sconv}Here and henceforth we use the convention that of objects ($A,B,gf$ elements, $C$ as set) like $A_{\pm}B_{\mp}$ or $g,f_{\pm}\in C^{\pm}$ are meant for the upper and lower signs separately if not otherwise stated, i.e. $A_{\pm}B_{\mp}$ is either $A_{+}B_{-}$ or $A_{-}B_{+}$ whereas $A_{\pm}B_{\pm}$ stands for either $A_{+}B_{+}$ or $A_{-}B_{-}$. In a similar manner we write $g,f_{\pm}\in C^{\pm}$ for either $g,f_{+}\in C^{+}$ or $g,f_{-}\in C^{-}$.}

\bnotecontent{F34}{We write $Q_{\pm\pm}$ for either $Q_{++}$ or $Q_{--}$ while $Q_{\pm\mp}$ is either $Q_{+-}$ or $Q_{-+}$. We apply this to all coming quantities with similar subscripts.}

\bnotecontent{F35}{Any properly discontinuous group actions on a metric space is strongly properly discontinuous; see \cite[Thm.1]{deovara}. This holds true for ordinary manifolds which comply the properties of being paracompact and Hausdorff. This implication follows as for discrete groups properness is equivalent to strong proper discontinuity; see \cite[Cor.3.22]{TomDieckTammo1938}.}